%% file: main.tex
\documentclass[a4paper,12pt,oneside,headsepline]{scrartcl}

\input{packages}

\input{layout}
\input{shortcuts}

\input{shortcuts_this}

\newcommand{\headertitle}{{%
  Formal deformations of modular forms and multiple~$\rmL$\nbd{}values
}}
\newcommand{\headerauthors}{%
  A.~Keilthy,
  M.~Raum
}
\title{%
  Formal deformations of modular forms and multiple~$\rmL$\nbd{}values
}
\author{%
Adam Keilthy%
\thanks{The author was supported by the ``Postdoctoral program in Mathematics for researchers from outside Sweden'' (project KAW 2020.0254)}
\and
Martin Raum%
\thanks{The author was partially supported by Vetenskapsr\aa det Grant~2019-03551 and~2023-04217.}%
}
\ifdraft{\date{\today\ at\ \thistime}}{\date{}}

\begin{document}

\thispagestyle{scrplain}
\begingroup
\deffootnote[1em]{1.5em}{1em}{\thefootnotemark}
\maketitle
\endgroup

\begin{abstract}
\small
\noindent
{\tbf Abstract:}
We relate analytically defined deformations of modular curves and modular forms from the literature to motivic periods via cohomological descriptions of deformation theory. Leveraging cohomological vanishing results, we prove the existence and essential uniqueness of deformations, which we make constructive via established Lie algebraic arguments and a notion of formal logarithmic deformations. Further, we construct a canonical and a totally holomorphic canonical universal family of deformations of modular forms of all weights, which we obtain from the canonical cocycle associated with periods on the moduli space~$\cM_{1,1}$. Our uniqueness statement shows that non-critical multiple $\rmL$-values, which  appear in our deformations but are a priori non-geometric, are genuinely linked to deformations. Our work thus suggests a new geometric perspective on them.
\\[.3\baselineskip]
\noindent
\textsf{\textbf{%
  non-critical multiple~$\rmL$\nbd{}values%
}}%
\noindent
\ {\tiny$\blacksquare$}\ %
\textsf{\textbf{%
  cohomological logarithmic formal deformations%
}}
\\[.2\baselineskip]
\noindent
\textsf{\textbf{%
  MSC Primary:
  11F67%
}}%
\ {\tiny$\blacksquare$}\ %
\textsf{\textbf{%
  MSC Secondary:
  11G55, 11F11%
}}
\end{abstract}

\vspace{-1.5em}
\renewcommand{\contentsname}{}
\setcounter{tocdepth}{1}
\tableofcontents
\setcounter{tocdepth}{2}
\vspace{1.5em}

\Needspace*{4em}
\phantomsection
\label{sec:introduction}
\addcontentsline{toc}{section}{Introduction}
\markright{Introduction}

\lettrine[lines=2,nindent=.2em]{\tbf T}{\,here} exists an emergent theory of period integrals on~$\cM_{1,1}$ with a particular focus on iterated Eichler integrals of modular forms, analogous to the study of periods on~$\cM_{0,n}$. Such integrals arise naturally in the study of modular graph functions and genus 1 Feynman amplitudes~\cite{broedel19, duhr23, tapuskovic23}, and their study has become the next step in understanding aspects of quantum field theory and its connections to the theory of motives.

In recent work by Bogo, we find a surprising appearance of such periods in the context of deformations of modular curves, which is a priori unrelated. While periods of~$\cM_{0,n}$ have appeared in the theory of deformation-quantisation, the geometric theories of motives and deformation have been largely unrelated in the literature. We see some connections between iterated integrals and the construction of Maurier-Cartan elements which encode deformations of Lie algebras~\cite{levin07,enriquez21}, and in the study of deformation-quantisation of Poisson brackets~\cite{banks20}, but neither explicitly arise via deformation of a variety.  Also of note is that the periods of Bogo's work are iterated integrals of cusp forms, whereas iterated integrals of Eisenstein series have been the primary examples in the literature thus far.

The goal of the present paper is to systematically explain Bogo's result, with specific reference to the appearance of periods in his work, from the perspective of cohomological deformation theory. We establish the existence and essential uniqueness of such deformations, as well as providing a universal and canonical deformation family

Bogo's work considers solutions to certain differential equations, determined recursively via iterated integrals. These differential equations arise from the theory of uniformisation, specifically that of the punctured sphere, and allow Bogo to construct deformations of modular forms in the case where the modular curve supports a weight~$1$ modular form. The appearance of periods in these deformations is not immediate, involving non-critical multiple~$\rmL$\nbd{}values, which \textit{a priori} lack a geometric interpretation. However, the transformation properties of these deformations of modular forms, at least in low degree, involve genuine periods given by critical multiple~$\rmL$\nbd{}values. This leads us to three key questions that we will answer in this work:
\begin{enumerate}[label=(Q\arabic*)]
\item
\label{it:mainquestion:why_what_periods}
Why do periods arise in deformations of modular curves, and which periods arise?
\item
\label{it:mainquestion:restrictions}
To what extent are the assumptions on genus~$0$ and weight~$1$ cusp forms required in the theory?
\item
\label{it:mainquestion:non_critical}
To what extent are the appearances of non-critical multiple~$\rmL$-values an artefact of Bogo's construction?
\end{enumerate}

A foundational step in the present work is to pass from analytic deformations to infinitesimal ones. We explain in detail in Section~\ref{ssec:bogo-def} that the first order analytic deformations employed by Bogo are in one-to-one correspondence with first-order deformations that we define in Section~\ref{ssec:first_order_deformations_modular_forms}. In cohomological deformation theory we have the additional concept of formal deformations, which we revise in Section~\ref{ssec:formal_deformations_modular_forms}. In general, formal deformations are more restrictive than first order deformations, however in our case the theory extends. We also define the notion of a deformation section, giving a map taking modular forms to objects that transform via deformation cocycles

\begin{mainexample}
Bogo's deformation operators give examples of first order deformation sections. For a cusp form~$h$ of weight~$4$, the associated deformation section induces the map on the space of weight~$k$ modular forms
\begin{gather*}
  f \lmto \frakp(f)
\defeq
  f + \big(2\tilde{h}f^\prime + k \tilde{h}^\prime f\big)\rho
\tx{,}
\end{gather*}
where $\tilde{h}$ is the holomorphic Eichler integral of the cusp form $h$. This first order deformation of $f$ satisfies 
\begin{gather*}
  \frakp(f)(\ga\tau)
=
  j_\ga(\tau)^k\, a_\ga(\tau,\partial_\tau)\, \frakp(f)(\tau)
\end{gather*}
for a linear cocycle $a_\bullet$. A formal prolongation~$\cP$ of~$\frakp$ is a formal series~$1 + \frakp \rho + \sum_{m=2}^\infty \frakp_m \rho^m$ in~$\rho$ with coefficients~$\frakp_m$ in linear holomorphic differential operators subject to highly restrictive conditions that we specify in Definition~\ref{def:algebraic_deformation_section}.
\end{mainexample}

The next statement is a reformulation of Theorems~\ref{thm:extension-of-first-order} and ~\ref{thm:existence_and_uniqueness_of_section}, which we prove in Sections~\ref{ssec:extension_first_order_deformations} and~\ref{ssec:existence_deformation_section}.

\begin{maintheorem}
\label{mainthm:extension_from_first_order}
Every first order algebraic deformation in the sense of Definition~\ref{def:first_order_deformation_cocycle} extends to a formal one in the sense of Definition~\ref{def:algebraic_deformation_cocycle}. Further, there is a unique formal deformation section in the sense of Definition~\ref{def:algebraic_deformation_section} associated to each of them.
\end{maintheorem}

Theorem~\ref{mainthm:extension_from_first_order} allows us to work with formal deformations, for which there are more adequate cohomological tools. The nature of these tools then suggests to further restrict the concept of formal deformations and work with formal logarithmic deformations. Throughout, this paper, including Definition~\ref{def:algebraic_deformation_cocycle} alluded to before, formal will be equivalent to formal logarithmic.

We can further restrict families of deformation. Since modular forms generated an algebra graded by weight, deformations for different weights should respect that grading and be compatible with products. Families of deformations with this property will be called universal. We study them in Section~\ref{ssec:universal_deformation}.

With these restrictions in place, the apparent connection between deformations and periods naturally raises the question whether there is an analogue of Brown's canonical cocycles. The next theorem provides an affirmative answer and asserts that the resulting deformation has all desired properties. It is a restatement of Theorem~\ref{thm:canonical-cocycle}, in which we provide expressions for the desired deformations.

\begin{maintheorem}
\label{mainthm:canonical-in-all-weight}
There exists an explicit ``totally holomorphic'' canonical universal family of deformations in all weights of motivic origin. The coefficients are given in terms of periods, specifically the critical multiple modular values
\begin{gather*}
  \Lambda_\bullet(h_1,\ldots,h_r;n_1,\ldots,n_r)
\end{gather*}
for~$(h_1,\ldots,h_r)$ cusp forms of weight~$4$, and~$0\leq n_i\leq 2$. The coefficients of the associated unique deformation section feature non-critical multiple modular values.
\end{maintheorem}

Theorem~\ref{mainthm:canonical-in-all-weight} answers Question~\ref{it:mainquestion:restrictions} in the most optimistic way: None of the assumptions imposed by Bogo are necessary. To some extent it also answers why periods enter the theory, and thus provides half an answer to Question~\ref{it:mainquestion:why_what_periods}. There is ambiguity left, however, as alternative formal deformations could extend a given first order one. This is addressed by the discussion following Remark \ref{rem:canonical-section}, giving an inexplicit cocycle in terms of periods and a recipe for constructing any deformation cocycle from this truly canonical cocycle via the transitive action of a monoid.
The next theorem is a summary of Theorems~\ref{thm:true-canonical-cocycle},~\ref{thm:transitive-action}, and extends the deformations of Theorem~\ref{mainthm:extension_from_first_order}.

\begin{maintheorem}
\label{mainthm:uniqueness}
There exists a canonical deformation cocycle of weight~$0$, whose coefficients are periods of the motivic fundamental group of the modular curve. Every other deformation cocycle of weight~$0$ is equivalent to the image under a~$\Ga$\nbd{}endomorphism of\/~$\bbD_0$. In particular, the totally holomorphic cocycle of Theorem~\ref{mainthm:canonical-in-all-weight} is obtained from the canonical by specialising a set of formal variables to~$0$. 
\end{maintheorem}

In conjunction with Theorem~\ref{mainthm:canonical-in-all-weight} and the uniqueness of deformation sections, we thus demonstrate that non-critical multiple modular~$\rmL$\nbd{}values genuinely appear in the deformation theory of modular curves.
We do not provide a geometric interpretation of this discovery, and leave this aspect to future work, but content ourselves with a brief remark on the matter:
While critical multiple~$\rmL$\nbd{}values can be linked to a geometric origin by a standard construction and have been studied in great detail (see for instance~\cite{brown17,dorigonietal2024}), only some non-critical ones enjoy such a connection so far, such as the expression of a non-critical~$\rmL$\nbd{}value for the unique weight $12$~$\SL{2}(\bbZ)$\nbd{}cusp form in terms of critical multiple~$\rmL$\nbd{}values associated to Eisenstein series due to Brown~\cite{brown17}. In this sense, the new angle offered by Theorems~\ref{mainthm:canonical-in-all-weight} and~\ref{mainthm:uniqueness} gives hope to identify a novel geometric source for non-critical multiple~$\rmL$\nbd{}values.

In relation to Bogo's work, we record that that he provides a formal deformation via a recursive procedure. It is not clear yet whether this procedure recovers a formal logarithmtic deformation that falls under Theorem~\ref{mainthm:uniqueness}.

Our work builds upon two key methods. We have already discussed the use of formal deformations in the context of Theorem~\ref{mainthm:extension_from_first_order}. Note that a formal approach is natural in the context of periods. It appears in the context of periods for~$\cM_{0,n}$, but also in Brown's work on multiple modular values associated with~$\cM_{1,1}$. We see similar use of formal techniques in Matthes' characterisation of the Kronecker function~\cite{matthes19}, Saad's work relating iterated Eisenstein integrals to multiple zeta values~\cite{saad20}, Gerkin et al.'s considerations of modular graph functions~\cite{gerken20}, and Schnep's work on elliptic associators~\cite{schneps23}. As such, our work on the deformation of modular curves is closely in line with an established language of periods.

A second key step in our treatment is the use of logarithmic deformations, which allows us to employ Lie algebraic techniques. Specifically, it enables the use of the Baker-Campbell-Hausdorff Formula. From a technical point of view, these Lie algebraic approa\-ches to calculation bring the present work much closer to work in motivic periods than to Bogo's. Furthermore, these methods will be applicable to variants of modular forms for which there is no analogue of uniformising differential equations, such as quasimodular forms and vector valued modular forms. We plan to address such extensions in a sequel.

\section{Preliminaries}%
\label{sec:preliminaries}

We briefly recall the basic notions of slash actions on functions and differential operators, deformations of (the corresponding) line bundles, and multiple modular values.

\subsection{Modular curves}%
\label{ssec:modular_curves}

Let~$\Ga\subset \SL2(\bbR)$ be a Fuchsian group, and let~$\Ga$ act on the upper half plane by M{\"o}bius transformations
\begin{gather*}
  \gamma\tau
:=
  \mfrac{a\tau +b}{c\tau + d}
\tx{,}\quad
  \ga = \begin{psmatrix} a & b \\ c& d \end{psmatrix}
\tx{.}
\end{gather*}

For~$\Ga$ acting freely and properly discontinuously on~$\HS$, the quotient space
\begin{gather*}
  Y_\Ga := \Ga\backslash\HS 
\end{gather*}
is an algebraic curve with fundamental group~$\Ga$~\cite{eilenberg45}. For a less well behaved action, we should consider the orbifold quotient~$\Ga\backslash\backslash\HS$ instead. In interest of simplicity, we shall not consider the orbifold situation, though all results should extend naturally.

For each~$\ga\in\Ga$ define 
\begin{gather*}
  j_\ga(\tau)
=
  c \tau + d
\tx{,}\quad
  \ga = \begin{psmatrix} a & b \\ c & d \end{psmatrix}
\tx{.}
\end{gather*}
We let~$L_k$ be the quotient of the trivial line bundle on~$\HS$ be the~$\Ga$\nbd{}action
\begin{gather*}
  \gamma\cdot(\tau,z)
=
  \big( \gamma\tau,\, j_\ga^k(\tau)z \big)
\tx{.}
\end{gather*}
Global sections of this line bundle define (weakly holomorphic) modular forms of weight~$k$. In classical language this corresponds to functions satisfying the modular invariance property
\begin{gather*}
  (c \tau + d)^{-k}\,
  f\big( \mfrac{a \tau + b}{c \tau + d} \big)
=
  f(\tau)
\tx{.}
\end{gather*}

\subsection{Functional multiple modular values}
\label{ssec:preliminaries:multi_modular}

Brown introduces a class of iterated integrals associated to modular forms, defining periods of the (relative completion of the) fundamental group of modular curves, that he calls multiple modular values~\cite{brown17}. Brown's integrals are valued in certain polynomial representations of the modular group, but for our purposes it will be convenient to slightly modify this definition, to produce explicit components of Eichler-Shimura integrals.

\begin{definition}
\label{def:func_mult_mod}
Given cusp forms~$h_1,\ldots,h_d$ of weight~$k_1,\ldots,k_d$, define the functional multiple modular value recursively by
\begin{gather*}
  \Lambda_\tau(h_1,\ldots,h_d;n_1,\ldots,n_d)
:=
  \int_\tau^{i\infty} h_1(z) z^n\, \Lambda_z(h_2,\ldots,h_d;n_2,\ldots,n_d) \dz
\end{gather*}
for integers~$0\leq n_i$ and~$\tau\in\ov{\HS}$. We define the empty integral by~$\Lambda_\tau(-;-):=1$.
\end{definition}

\begin{remark}
Note that these are components of more classical Eichler-Shimura integrals, as in~\cite{brown17}. It is also worth noting that Brown's definition of multiple modular values \textit{a priori} only allows for the parameters~$n_i$ to vary~$0\leq n_i <k_i-2$. For our purpose, this extended formulation, following the convention of Matthes~\cite{matthes22}, is more convenient.
\end{remark}

\begin{remark}
Brown defines functional multiple modular values for all modular forms, handling the divergent integrals arising via a regularisation procedure~\cite{brown17,brown19}. As we do not explicitly consider such integrals in this article, we do not give the full definition here.
\end{remark}

Classical periods arise as special values of these multiple modular values~\cite{brown17,brown19}. Note that the notation~$\Lambda_\bullet$ introduced in the next definition strictly speaking overlaps with the one in Definition~\ref{def:func_mult_mod}, however we expect that there will be no confusion between subscripts~$\tau \in \HS$ and~$\ga \in \Ga$.

\begin{definition}
\label{def:multi_modular_value}
Given cusp forms~$h_1,\ldots,h_d$ of weight~$k_1,\ldots,k_d$, define the (classical) multiple modular value as the function
\begin{gather*}
  \Lambda_\bullet(h_1,\ldots,h_d;n_1,\ldots,n_d)
\defcol
  \Ga \lra \bbC
\end{gather*}
given by the integral
\begin{gather*}
  \Lambda_\ga(h_1,\ldots,h_d;n_1,\ldots,n_d)
\defeq
  \Lambda_{\ga^{-1}\infty}(h_1,\ldots,h_d;n_1,\ldots,n_d)
\tx{.}
\end{gather*}
\end{definition}

For example, the coefficients of the period polynomial associated to a weight\nbd{}$k$ cusp form~$h$ are precisely the values of the multiple modular value~$\Lambda_{\ga}(h;n)$ for~$0 \le n \le k-2$, and are closely related to special values of the L-function associated to~$h$.

\subsection{The differential algebra as a~\texpdf{$\Ga$}{Gamma}-module}
\label{ssec:preliminaries:differential_algebras}

The slash action on functions~$\HS \ra \CC$ has an analogue for differential operators. Consider the~$\bbC$\nbd{}vector space~$\cO[\partial_\tau]$ of finite order linear differential operators with holomorphic functions on~$\HS$ as coefficients. This may be equipped with an action of~$\Ga$ by conjugation with the weight cocycle
\begin{gather}
\label{eq:def:double_slash_action}
  \big( a \big\|_k \ga \big) (\tau, \partial_\tau)
\defeq
  j^{-k}_\ga(\tau)\,
  a\big( \ga \tau, \partial_{\ga \tau} \big)\,
  j^k_\ga(\tau)
\tx{.}
\end{gather}
In this context, we record that
\begin{gather*}
  \partial_{\ga \tau}
=
  (c \tau + d)^2\,
  \partial_\tau
\tx{.}
\end{gather*}
We write~$\ord(a)$ for the order, as a differential operator, of an element~$a\in\bbD_k$.

Note that this action extends linearly to the ring of formal series
\begin{gather*}
  \cO[\partial_\tau]\llbrkt \rho_1, \ldots, \rho_n \rrbrkt
\end{gather*}
over~$\cO[\partial_\tau]$.

Since we will consider deformations in several parameters, it is convenient to write~$|m| = m_1 + \cdots + m_n$ for~$n$\nbd{}tuples~$m$ of numbers.

\subsection{Deformations of line bundles}
\label{ssec:preliminaries:deformations}

We will construct deformations of modular forms by constructing sections of a deformation of the line bundle~$L_k$.
Following Sernesi~\cite{sernesi06} we define a deformation of a line bundle as follows. Note that we take a deformation of a line bundle as the simultaneous deformation of a line bundle and the base scheme (\cite{sernesi06}[Section 3.3.3] ), not just the deformation of a line bundle over a fixed base scheme (\cite{sernesi06}[Section 3.3.1] ), as such deformations are uninteresting in our case.

\begin{definition}
\label{def:first_order_deformation_line_bundle}
Let~$X$ be a scheme over a field~$K$, and~$L$ a line bundle on~$X$. A first order deformation of~$L$ consists of a pair~$(\cX,\cL)$ where
\begin{gather*}
\begin{tikzpicture}
  \matrix (m) [matrix of math nodes,
    ampersand replacement=\&,
    column sep = 2.5em, row sep = 2em
    ]
  { X        \& \cX \\
    \Spec(K) \& \Spec\big( K[\rho_1,\ldots,\rho_n] \big/ (\rho)^2  \big) \\
  };
  \path[-stealth]
  (m-1-1) edge (m-1-2)
  (m-2-1) edge (m-2-2)
  (m-1-1) edge (m-2-1)
  (m-1-2) edge node[right] {$\pi$} (m-2-2);
\end{tikzpicture}
\end{gather*}
is a commutative diagram in which~$\pi$ is flat and surjective,~$\cL$ is a line bundle on~$\cX$ such that~$L=\cL|_X$, and~$(\rho)$ denotes the ideal generated by~$\rho_1,\ldots,\rho_n$.
\end{definition}

Traditionally, a first order deformation has~$n=1$, but it will be convenient for us to consider this slightly wider definition.

\begin{definition}
\label{def:formal_deformation_line_bundle}
Let~$X$ be a scheme over a field~$K$, and~$L$ a line bundle on~$X$. A formal deformation of~$L$ consists of a pair~$(\cX,\cL)$ where
\begin{gather*}
\begin{tikzpicture}
  \matrix (m) [matrix of math nodes,
    ampersand replacement=\&,
    column sep = 2.5em, row sep = 2em
    ]
  { X        \& \cX \\
    \Spec(K) \& \Spec\big( K\llbrkt \rho_1,\ldots,\rho_n \rrbrkt  \big) \\
  };
  \path[-stealth]
  (m-1-1) edge (m-1-2)
  (m-2-1) edge (m-2-2)
  (m-1-1) edge (m-2-1)
  (m-1-2) edge node[right] {$\pi$} (m-2-2);
\end{tikzpicture}
\end{gather*}
is a commutative diagram in which~$\pi$ is flat and surjective, and~$\cL$ is a line bundle on~$\cX$ such that~$L=\cL|_X$.
\end{definition}

\begin{definition}
\label{def:deformation_isomorphism}
We say two first order (formal) deformations~$(\cX,\cL)$ and~$(\wtd{\cX},\wtd{\cL})$ of~$(X,L)$ are isomorphic if there exists an isomorphism of schemes~$f:\cX\cong\wtd{\cX}$ such that~$\cL\cong f^*\wtd{\cL}$, and~$f$ is compatible with the commutative diagrams.
\end{definition}

When~$X$ is a nonsingular projective variety, the isomorphism classes of first order deformations are well understood. Denote by~$\cO_X$ the structure sheaf on~$X$, and by~$T_X$ the tangent sheaf. Then, if~$L$ is a line bundle on~$X$, the Chern class of~$L$ defines an extension
\begin{gather*}
  0 \lra \cO_X \lra \cE_L \lra T_X \lra 0
\tx{,}
\end{gather*}
where~$\cE_L$ is unique up to isomorphism, and referred to as the Atiyah extension associated to~$L$.

Recall the Kodaira-Spencer Theorem, which associates to a deformation of~$X$ a class in~$\rmH^1(X,T_X)$. We refer to it as the Kodaira-Spencer class of the deformation. We then have the following result~\cite{sernesi06}.

\begin{theorem}
\label{thm:first-order-class}
Let~$(X,L)$ be a pair consisting of a nonsingular projective algebraic variety~$X$ and a line bundle~$L$ on~$X$. Then there is a canonical isomorphism between the space of first order deformations of~$(X,L)$ with~$n=1$ and the sheaf cohomology~$\rmH^1(X,\cE_L)$. Furthermore, the projection to~$\rmH^1(X,T_X)$ is precisely the Kodaira-Spencer class of the deformation of~$X$.
\end{theorem}

For our more general definition of a first order deformation, an analogous result holds, with first order deformations being encoded by~$n$-tuples of elements of~$\rmH^1(X,T_X)$.

\section{First order deformations and analytic deformations}
\label{sec:first-order-deformations}

The goal of this section is to give a rigorous definition of first order deformations of modular forms in Section~\ref{ssec:first_order_deformations_modular_forms}, and then connect it Bogo's results in Section~\ref{ssec:bogo-def}. We require a specific differential Lie algebra that we examine and justify in Lemma~\ref{lem:algebraic-submodule} and the following remark. It leads naturally to deformation cocycles in Definition~\ref{def:first_order_deformation_cocycle}, which in turn come with deformation sections as specified in Definition~\ref{def:first_order_deformation_section}. Deformations of modular forms then arise from the these deformation sections. A classification of first order deformation cocycles is possible and revisited in Proposition~\ref{prop:first-order-cusp-forms}.

\subsection{Cocycles associated with deformations of line bundles}%
\label{ssec:first_order_cocycles_line_bundle}

For each weight~$k$, we first define a~$\Ga$\nbd{}submodule of~~$\cO[\partial_\tau]$ equipped with the weight~$k$ action by
\begin{gather}
\label{eq:def:frako}
  \frako_k
\defeq
  \cO\partial_\tau \oplus \cO
\tx{.}
\end{gather}
In the ensuing discussion, the first direct summand matches the tangent bundle in the Atiyah bundle, and the second one corresponds to~$\cO_X$. That is, by Theorem~\ref{thm:first-order-class}, the first term will correspond to deformations of the base, while the second one corresponds to deformations of the line bundle.

At times, we will write, say,~$\frakj(\tau, \partial_\tau)$ for the elements of~$\frako_k$ in order to emphasise that they are differential operators.

A first order deformation of the modular curve~$Y_\Ga$ and the line bundle~$L_k$ as in Definition~\ref{def:first_order_deformation_line_bundle} is determined, up to isomorphism, by a~$n$-tuple of elements of~$\rmH^1(Y_\Ga,\cE_{L_k})$. By pulling back to~$\HS$, and identifying the sheaf cohomology of~$Y_\Ga$ with the group cohomology of~$\Ga$~\cite{eilenberg45}, we easily find that
\begin{gather*}
  \rmH^1(Y_\Ga,\cE_{L_k}) \cong \rmH^1(\Ga,\frako_k)
\tx{.}
\end{gather*}

Indeed, if we have a first order deformation~$(\cY_\Ga,\cL_k)$ of~$(Y_\Ga,L_k)$. By pulling~$\cL_k$ back to a line bundle on~$\HS$, we must have that~$\cL_k$ is a quotient of the trivial line bundle, and as such the sections of~$\cL_k$ are can be identified with functions
\begin{gather*}
  f \defcol \HS \lra \bbC [ \rho_1,\ldots,\rho_n ] \big\slash (\rho)^2
\end{gather*}
such that, suppressing dependency on~$k$,
\begin{gather*}
  f(\ga\tau) = \frakj_\ga(\tau)(\tau,\partial_\tau)\, f(\tau)
\end{gather*}
for some linear differential operator
\begin{gather*}
  \frakj_\ga(\tau, \partial_\tau)
\defcol \cO \lra \cO[ \rho_1,\ldots,\rho_n ] \big\slash (\rho)^2
\quad\tx{with\ }
  \frakj_\ga(\tau,\partial_\tau)
\equiv
  j_\ga^k(\tau)
  \;\pmod{(\rho)}
\end{gather*}
that satisfies
\begin{gather*}
  \frakj_{\ga\delta}(\tau,\partial_\tau)
=
  \frakj_\ga(\delta\tau, \partial_{\delta \tau})\, \frakj_\delta(\tau,\partial_\tau)
\quad
  \tx{for all\ }\ga, \delta \in \Ga
\tx{.}
\end{gather*}

Writing 
\begin{gather}
\label{eq:first_order_deformation_cocycle_line_bundle_decomposition}
  \frakj_\ga(\tau, \partial_\tau)
=
  j_\ga(\tau, \partial_\tau)\,
  \Big(1 + \sum_{i=1}^n \frakj_{\ga,i}(\tau, \partial_\tau) \rho_i \Big)
\tx{,}
\end{gather}
the above compatibility condition reduces using~\eqref{eq:def:double_slash_action} to 
\begin{gather}
\label{eq:first_order_deformation_cocycle_line_bundle_relations}
  \frakj_{\ga\delta, i}
=
  \frakj_{\ga,i}\big\|_k \delta + \frakj_{\delta,i}
\quad
  \tx{for all\ }
  \ga,\delta\in\Ga
\tx{.}
\end{gather}
This is precisely to say that each~$\frakj_{\bullet,i}$ defines a~$1$-cocycle valued in~$\frako_k$ and hence an element of~$\rmH^1(\Ga,\frako_k)$.

However, the full deformation theory of modular curves with coefficients in holomorphic function is fairly trivial, so we don't obtain much of interest. As such, we will restrict to algebraic deformations, that is to say to polynomial coefficients. The next lemma provides a suitable coefficient module.

\begin{lemma}
\label{lem:algebraic-submodule}
Denote by~$\bbC[\tau]_{\leq 2}$ the space of polynomials in~$\tau$ of degree at most two. Then the vector subspace 
\begin{gather}
  \frakd_k
\defeq
  \big\{ 2p(\tau)\partial_\tau + kp^\prime(\tau) \condcol p\in\bbC[\tau]_{\leq 2} \big\}
\end{gather}
of\/~$\frako_k$ is a~$\Ga$\nbd{}submodule.
\end{lemma}

\begin{remark}
Note that given a holomorphic modular form~$h$ the map 
\begin{gather*}
  f \lmto 2\tilde{h}f^\prime + k \tilde{h}^\prime f
\end{gather*}
acts as a derivation on the (weight graded) ring of modular forms~$M(\Ga)$. It is for this reason that~$\frakd_k$ is defined as it is. Had we defined it naively as the largest submodule of~$\frako_k$ contained in~$\bbC[\tau]\partial_\tau\oplus\bbC[\tau]$, there would exist~$1$-cocycles that did not correspond to derivations, given by an element of~$\rmZ^1(\Ga,\frakd_k)$ plus a homomorphism~$\Ga\to\bbC$. Deformations associated with the latter are studied in the theory of higher order modular forms~\cite{goldfeld-2002,chinta-diamantis-osullivan-2002}.
\end{remark}

\begin{proof}
This is a direct computation. Let~$\ga\in\Ga$ be given by the matrix~$\ga = \begin{psmatrix} a & b \\ c& d \end{psmatrix}$. Then 
\begin{align*}
  \big( 2p(\tau)\partial_\tau + kp^\prime(\tau) \big)\big\|_k\ga
&=
  2j_\ga^{-k}(\tau)p(\ga\tau)\partial_{\ga\tau}j_\ga^k(\tau) + k p^\prime(\ga\tau)
\\
&=
  2j_\ga^2(\tau)p(\ga\tau)\partial_\tau + 2kj_\ga^\prime(\tau)j_\ga(\tau)p(\ga\tau)+kp^\prime(\ga\tau)
\tx{.}
\end{align*}
Suppose~$p(\tau)=p_2\tau^2 + p_1\tau+ p_0$. Then the coefficient of~$\partial_\tau$ is given by
\begin{gather*}
  2 \big(a^2 p_2+ac p_1+c^2 p_0\big)\tau^2 + 2\big(2ab p_2 + ad p_1 + bc p_1 + 2cd p_0\big)\tau + 2\big(b^2 p_2 + bd p_1 + d^2 p_0\big)
\tx{,}
\end{gather*}
while the purely polynomial term reduces, using~$ad-bc=1$, to
\begin{gather*}
  2k \big(a^2 p_2 + ac p_1 + c^2 p_0 \big)\tau + k\big( 2ab p_2 + ad p_1 + bc p_1 + 2cd p_0 \big)
\tx{,}
\end{gather*}
which easily seen to be an element of~$\frakd_k$.
\end{proof}

For later purpose, we record that we may equip both~$\frako_k$ and~$\frakd_k$ with the structure of a Lie algebra as follows. Suppose 
\begin{gather*}
  a(\tau) = f(\tau)\partial_\tau + g(\tau)
\quad\tx{and}\quad
  b(\tau) = r(\tau)\partial_\tau + s(\tau)
\end{gather*}
are elements of~$\frako_k$. Then define
\begin{gather}
\label{eq:def:lie_bracket_ok}
  [a,b](\tau)
\defeq
  \big( f(\tau)r^\prime(\tau) - r(\tau)f^\prime(\tau) \big)\,\partial_\tau
+ f(\tau)s^\prime(\tau) - r(\tau)g^\prime(\tau)
\end{gather}
as the commutator of~$a$ and~$b$ as differential operators. It is easy to check that this restricts to a Lie bracket on~$\frakd_k$, as
\begin{gather*}
  \big[ 2p\partial_\tau + kp^\prime,2q\partial_\tau + kq^\prime \big]
=
  4\big( pq^\prime - p^\prime q \big)\,\partial_\tau
+ 2k\big( pq^{\prime\prime} - p^{\prime\prime}q \big)
\end{gather*}
and~$pq^\prime - p^\prime q$ is a quadratic polynomial. Furthermore, direct computation shows that
\begin{gather}
  \big[a\big\|_k \ga, b\big\|_k \ga \big] = [a,b]\big\|_k\ga
\tx{.}
\end{gather}

\subsection{First order deformations of modular forms}%
\label{ssec:first_order_deformations_modular_forms}

We will now see that first order deformations of modular forms arise from the action of deformation sections, which are~$0$\nbd{}cycles taking values in differential operators with holomorphic coefficients. To identify meaningful~$0$\nbd{}cycles, we choose those that trivialise a~$1$\nbd{}cocycle taking values in differential operators with polynomial coefficients. This justifies our first Definition~\ref{def:first_order_deformation_cocycle}.

We emphasise how this approach parallels the relation between period polynomials and Eichler integrals. The period polynomial~$p_f$ associated with a weight\nbd{}$k$ cusp form~$f$ for, say,~$\SL{2}(\ZZ)$ can be identified with a suitable~$1$\nbd{}cocycle of~$\SL{2}(\ZZ)$. In this language, the Eichler integral~$\cE_f$ associated with~$f$ are naturally~$0$\nbd{}cocycles that trivialise them, as reflected by the coboundary equation
\begin{gather*}
  \cE_f \big|_{2-k}\,\begin{psmatrix} 0 & -1 \\ 1 & 0 \end{psmatrix}
  -
  \cE_f
=
  p_f(\tau)
\tx{.}
\end{gather*}

\begin{definition}
\label{def:first_order_deformation_cocycle}
A first order algebraic deformation cocycle of weight~$k$ is an element
\begin{gather*}
  a_\bullet
\in
  \rmZ^1\big( \Ga,\, \frakd_k[\rho_1, \ldots, \rho_n] \big\slash (\rho)^2 \big)
\quad\tx{with\ }
  a_\bullet(\tau,\partial_\tau)
\equiv
  0
  \;\pmod{(\rho)}
\tx{.}
\end{gather*}
We identify it with a collection of~$n$ maps
\begin{gather*}
  a_{\bullet,i} \defcol \Ga \lra \frakd_k
\quad\tx{such that}\quad
  a_{\ga\delta,i} = a_{\ga,i}\big\|_k \delta + a_{\delta,i}
\quad
  \tx{for all\ }\ga, \delta \in \Ga
\end{gather*}
via the equation~$a_\bullet = \sum_{i=1}^n a_{\bullet,i} \rho_i$.

We say two deformation cocycles~$a_{\bullet}$,~$\tilde{a}_{\bullet}$ are equivalent if they define the same class in the group cohomology~$\rmH^1(\Ga, \frakd_k[\rho_1,\ldots,\rho_n] \slash (\rho)^2)$. This is to say,
\begin{gather*}
  \tilde{a}_{\ga,i}
=
  a_{\ga,i}
  \,+\,
  c_i\big\|_k\ga - c_i^{-1}
\qquad
  \tx{for some\ }
  c_i\in\frakd_k
\quad
  \tx{and for all~$i=1,\ldots,n$.}
\end{gather*}
We call the image of a first order algebraic deformation cocycle in cohomology a first order algebraic deformation class.
\end{definition}

We will suppress the emphasis on algebraic and refer to cocycles in Definition~\ref{def:first_order_deformation_cocycle} as first order deformation cocycles.

Observe the second description of first order deformation cocycles in Definition~\ref{def:first_order_deformation_cocycle} parallels the cocycle relation~\eqref{eq:first_order_deformation_cocycle_line_bundle_relations} satisfied by~$\frakj_\bullet$ associated with first order deformations of line bundles. In Proposition~\ref{prop:equivalent-cocycles} we show that equivalent deformation cocycles define isomorphic deformations of\/~$L_k$.

Next we look at the sections of the deformed line bundle. 

\begin{definition}
\label{def:first_order_deformation_section}
A first order deformation section of weight~$k$ and deformation cocycle~$a_{\bullet}$ is a trivialisation~$b(\tau,\partial_\tau)$ under the coboundary map
\begin{align*}
  \rmZ^0\big( \Ga,\, \frako_k[\rho_1,\ldots,\rho_n] \big\slash (\rho)^2 \big)
&
\lra{}
  \rmZ^1\big( \Ga,\, \frako_k[\rho_1,\ldots,\rho_n] \big\slash (\rho)^2 \big)
\supseteq
  \rmZ^1\big( \Ga,\, \frakd_k[\rho_1,\ldots,\rho_n] \big\slash (\rho)^2 \big)
\\
  \frakb
&\lmto
  a_\bullet
\tx{.}
\end{align*}
We identify it via the equation~$b = \sum_{i=1}^n b_i \rho_i$ with a collection of~$n$ elements
\begin{gather*}
  b_i \in o_k
\quad\tx{such that}\quad
  b_i\big\|_k\ga - b_i = a_{\ga,i}
\qquad
  \tx{for all~$i=1,\ldots,n$ and all~$\ga\in\Ga$.}
\end{gather*}
\end{definition}

A deformation section~$b$ of weight~$k$ and cocycle~$a_\bullet$ provides us with a deformation
\begin{gather}
\label{eq:first_order_deformation_modular_forms}
  \Big( \exp\big( b(\tau, \partial_\tau) \big) (f) \Big)(\ga\tau)
=
  \big( 1 + b(\tau, \partial_\tau) \big)(f)(\ga\tau)
\end{gather}
of any modular form~$f$ via the action of elements of~$\frako_k$ as differential operators. Note that this implies that, for a modular form~$f$ of weight~$k$, the deformed modular form transforms as
\begin{gather*}
  \big( 1 + b(\tau, \partial_\tau) \big)(f)(\ga\tau)
=
  j_\gamma^{k}(\tau) \Big(1+\sum_{i=1}^n a_{\ga,i}\Big)\, \big( b_i(\tau,\partial_\tau) f \big)(\tau)
\tx{.}
\end{gather*}
In particular, the~$a_{\ga,i}$ on the right hand side recover the cocycles~$\frakj_{\ga,i}$ in~\eqref{eq:first_order_deformation_cocycle_line_bundle_decomposition}, and~$b$ indeed provides a deformation of~$L_k$ of class~$a_\bullet$.

As it turns out, the classification of first order algebraic deformations, and choices of deformation section, are well known.

\begin{proposition}
\label{prop:first-order-cusp-forms}
First order algebraic deformations are in bijection with~$n$-tuples of elements of~$\rmM_4(\Ga)\oplus \overline{\rmS_4(\Ga)}$.
\end{proposition}
\begin{proof}
As every element of~$\frakd_k$ is uniquely determined by the coefficient of~$\partial_\tau$, we have an isomorphism
\begin{gather*}
  \rmH^1(\Ga,\frakd_k)
\cong
  \rmH^1\big( \Ga, \bbC[\tau]_{\leq 2} \big)
\tx{,}
\end{gather*}
where we equip~$\bbC[\tau]_{\leq 2}$ with the~$\Ga$\nbd{}action
\begin{gather*}
  p(\tau) \big\| \ga
\defeq
  j_\ga(\tau)^2\, p(\ga\tau)
\tx{.}
\end{gather*}
However, by the Eichler-Shimura isomorphism~$\rmH^1(\Ga,\bbC[\tau]_{\leq 2})$ is isomorphic to~$\rmM_4(\Ga)\oplus \overline{\rmS_4(\Ga)}$.
\end{proof}

The polynomials classifying first order deformations are called \emph{period polynomials}. For cusp forms, the associated quadratic polynomials is given by
\begin{gather*}
  j_\ga(\tau)^2\,
  \tilde{h}(\ga\tau) - \tilde{h}
\tx{,}
\end{gather*}
where
\begin{gather*}
  \tilde{h}(\tau)
\defeq
  \int_{\tau}^{i\infty} h(z)(\tau-z)^2 \,\dz
\end{gather*}
is the holomorphic Eichler integral of the cusp form~$h$. More generally, a period polynomial will be obtained by complex conjugation of coefficients, or a regularised Eichler integral of an Eisenstein series. In order to keep our computations concrete, we will only work with examples of first order deformations of the form
\begin{gather*}
  (b f)(\tau)
=
  f(\tau)
  +
  \sum_{i=1}^n
  \big( 2\tilde{h}(\tau) f^\prime(\tau) + k \tilde{h}^\prime(\tau) f(\tau) \big)\, \rho_i
\end{gather*}
for an~$n$-tuple of cusp forms~$(h_1,\ldots,h_n)$.

\begin{remark}
Note that, in the case where~$\Ga$ admits a weight~$1$ modular form, this is precisely the first order deformation constructed via uniformizing differential equations in~\cite{bogo21} discussed in the following section.
\end{remark}

\subsection{Bogo's deformations}%
\label{ssec:bogo-def}

We express Bogo's deformation operators in terms of deformation cocycles. This confirms that when Bogo's assumptions are satisfied, his constructions yields exactly the first order deformations in the sense of Definition~\ref{def:first_order_deformation_cocycle}.

In contrast to algebraic construction of deformations in Section~\ref{ssec:first_order_deformations_modular_forms}, Bogo constructs his deformations via the theory of uniformisation~\cite{bogo21,hempel88}. More precisely, for the Riemann sphere punctured at~$(\alpha_1,\ldots,\alpha_{n-3},0,1,\infty)$, Bogo considers the differential equation
\begin{gather}
\label{eq:bogo_differential_equation}
  P(t)\, \mfrac{\rmd^2 y}{\dt^2}
+ P^\prime(t)\,\mfrac{\rmd y}{\dt}
+ \sum_{i=1}^{n-3} \rho_i t^i\,y
+ \big(\tfrac{n}{2}-1\big)^2 t^{n-3}\,y
= 0
\tx{,}
\end{gather}
where 
\begin{gather*}
  P(t)=t(t-1)\prod_{i=1}^{n-3}(t-\alpha_i)
\tx{.}
\end{gather*}
There exists a unique set of values of the parameters~$(\rho_1,\ldots,\rho_{n-3})$ such that the solutions of this differential equation uniquely determine a Fuchsian group~$\Ga$ and a weight~$0$ modular function~$t(\tau)$ (which we a priori distinguish from variable~$t$ in~\eqref{eq:bogo_differential_equation}) such that
\begin{gather*}
  \Ga \big\backslash \HS
\cong
  \bbP^1 \setminus \{\alpha_1,\ldots,\alpha_{n-3},0,1,\infty\}
\tx{.}
\end{gather*}
These unique values are called the accessory parameters of the Riemann surface, and are notoriously difficult to determine, with only a few examples known exactly~\cite{zagier09}.

Furthermore, by composing the unique holomorphic solution~$y(t)$ with~$t(\tau)$, we obtain a square root of a weight~$2$ modular form for the action of~$\Ga$, which we denote by~$f(\tau)$. Note that the space of modular forms with transformation group~$\Ga$ is generated by~$t(\tau)$ and~$f(\tau)$

By considering solutions of the above differential equation for varying~$\rho = (\rho_1,\ldots,\rho_{n-3})$, Bogo similarly constructs functions~$t_\rho(\tau), f_\rho(\tau)$, which he views as deformations of~$t(\tau), f(\tau)$. From these, he constructs deformations of modular forms of arbitrary weight: given a weight~$k$ modular form~$g$, there exists a unique polynomial~$p(x)$ such that
\begin{gather*}
  g(\tau) = f(\tau)^k\, p\bigl(t(\tau)\bigr)
\tx{.}
\end{gather*}
The deformation~$g_\rho(\tau)$ of~$g$ is then defined to be
\begin{gather*}
  g_\rho(\tau) = f_\rho(\tau)^k\, p\bigl(t_\rho(\tau)\bigr)
\tx{.}
\end{gather*}
We then recover a first order deformation of~$g$ as the first derivatives of~$g_\rho$ in~$\rho$ at the accessory parameter.

The first order theory of Bogo's deformations can be related to the classical theory of first order deformations and thus to Definition~\ref{def:first_order_deformation_cocycle} as follows. Greater detail may be found in Chapter~VI of~\cite{ahlfors66}, particularly the final section.

Denote by~$\cT(\Gamma)$ the Teichm{\"u}ller space of~$Y_\Ga$, where~$\Ga$ is given by~$(\alpha_1, \ldots, \alpha_{n-3})$ as explained before. This is defined to be the moduli space of complex structures on the curve underlying the Riemann surface~$Y_\Ga$. The holomorphic tangent space to~$\cT(\Ga)$ at~$Y_\Ga$ is the space of harmonic Beltrami differentials
\begin{gather*}
  \big\{
  \mu \defcol
  \HS \ra \bbC
  \condcol
  \mu \mfrac{\rmd\ov{z}}{\dz}\text{\ is invariant under the }\Ga\text{-action}
  \big\}
\tx{,}
\end{gather*}
which is here isomorphic to the space weight~$4$ cusp forms~$\rmS_4(\Ga)$. In Bogo's language, it corresponds to the tangent space in the parameters~$(\rho_1, \ldots, \rho_{n-3})$ at the accessory parameter. Proposition~\ref{prop:first-order-cusp-forms} relates these cusp forms to algebraic deformations in Definition~\ref{def:first_order_deformation_cocycle}. To bridge Bogo's construction and ours, we determine the asssignment of cusp forms to first order deformation cocycles.

Given an element~$q(\tau)$ of~$\rmS_4(\Ga)$, denote by~$\nu(\tau) \defeq \Im(\tau)^2\, \ov{q}(\tau)$ the associated Beltrami differential and extend this to a function on~$\bbC$ by reflection across the real line. For~$\varepsilon \in \RR$, denote by~$f^{\varepsilon\nu} : \CC \ra \CC$ the unique solution of
\begin{gather*}
  \mfrac{\partial f^{\varepsilon\nu}}{\partial\ov{z}}
=
  \varepsilon\nu(z)\, \mfrac{\partial f^{\varepsilon\nu}}{\partial z}
\end{gather*}
that fixes~$0$,~$1$, and~$\infty$. Denote by~$\Ga^{\varepsilon\nu} \defeq f^{\varepsilon\nu}\; \Ga\; (f^{\varepsilon\nu})^{-1}$ the Fuchsian group associated to~$\varepsilon\nu$, and by~$t^{\varepsilon\nu}$ the normalised Hauptmodul for~$\Ga^{\varepsilon\nu} \backslash \bbH$. This defines a real-analytic function in~$\varepsilon$ such that~$t^{\varepsilon\nu}(\tau)|_{\varepsilon=0} = t(\tau)$, where~$t(\tau)$ is the modular function on~$\Ga \backslash \HS$ from Bogo's work.

Recall from~\eqref{eq:def:frako} and Lemma~\ref{lem:algebraic-submodule} that the coefficient module~$\frakd_k$ of first order deformation cocycles subsumes the tangent space~$T_{Y_\Ga}$. The  deformation cocycle associated with~$t^{\varepsilon\nu}$, and thus with~$q(\tau)$, is given by comparing the first order expansion in~$\ov{\varepsilon}$ in the various charts of our deformed Riemann surface. Via a standard procedure relating the cohomology of~$Y_\Ga$ to the group cohomology of~$\Ga$, we find that the deformation cocycle is given (up to a scalar) by the element~$\xi_\ga$ of~$\rmH^1(\Ga,T_{Y_\Ga}) \subset \rmH^1(\Ga, \frakd_0)$ defined by
\begin{gather*}
  \xi_\ga\big(t(\tau)\big)
=
  \mfrac{\partial}{\partial \ov{\varepsilon}} t^{\varepsilon\nu}(\ga\tau)
- \mfrac{\partial}{\partial \ov{\varepsilon}} t^{\varepsilon\nu}(\tau)
\tx{.}
\end{gather*}
Bogo's computation shows that~$\xi_\ga \in \bbC p_\gamma(\tau)\partial_\tau$ where~$p_\gamma$ is the period polynomial associated to the weight~$4$ cusp form~$q(\tau)$. Bogo's method of computation normalises to~$\xi_\ga = 2 p_\ga(\tau) \partial_\tau$. Thus, we recover the first-order algebraic deformations in weight~$0$ derived previously. Specifically, when identifying the tangent space to Bogo's parameters~$(\rho_1, \ldots, \rho_{n-3})$ with~$\rmS_4(\Ga)$, his deformation translates to our language as
\begin{gather*}
  \rmS_4(\Ga)
\lra
  \rmH^1(\Ga, \frakd_0)
\tx{,}\quad
  q
\lmto
  2 p_\ga(\tau)\, \partial_\tau
\tx{.}
\end{gather*}

More generally, the weight~$k$ deformation can be recovered by the isomorphism in Proposition~\ref{prop:weight-indep} in the next section.

\section{Formal deformations}%
\label{sec:formal_deformations}

In this section, we pursue two goals. First, we establish the analogue of Section~\ref{ssec:first_order_deformations_modular_forms} on first order deformations of modular forms in the context of formal deformations. This requires us to set up the analogues~$\bbO_k$ and~$\bbD_0$ of the Lie algebras~$\frako_k$ and~$\frakd_k$. They will serve in Section~\ref{ssec:formal_deformations_modular_forms} as the value modules of formal deformation cocycles and formal deformation sections.

Second, we relate the formal deformations to first order ones. Every formal deformation yields a first order deformation by truncation. In the subject of deformations, however, not every first order deformation affords an extension to a formal one, and this is an important part of the theory. We show in Section~\ref{ssec:extension_first_order_deformations} that in our setting every first order deformation extends. We give an existence argument leveraging vanishing of a suitable second cohomology group. It is not explicit, and therefore in Section~\ref{ssec:calculation_of_second_order_deformations} we illustrate how the extension to second order can be calculated.

\subsection{Cocycles associated with deformations of line bundles}%
\label{ssec:formal_cocycles_line_bundle}

Suppose that we have constructed a formal deformation~$(\cY_\Ga,\cL_k)$ of~$(Y_\Ga,L_k)$ as in Definition~\ref{def:formal_deformation_line_bundle}. In analogy with first order deformations, by pulling~$\cL_k$ back to a line bundle on~$\HS$, we must have that~$\cL_k$ is the quotient of the trivial line bundle, and as such the sections of~$\cL_k$ are can be identified with functions
\begin{gather*}
  f \defcol \HS \lra \bbC \llbrkt \rho_1,\ldots,\rho_n\rrbrkt
\end{gather*}
such that
\begin{gather*}
  f(\ga\tau) = \cJ_\ga(\tau)\, f(\tau)
\end{gather*}
for some linear differential operator
\begin{gather*}
  \cJ_\ga(\tau, \partial_\tau) \defcol \cO \lra \cO\llbrkt \rho_1,\ldots,\rho_n\rrbrkt
\quad
  \tx{with\ }
  \cJ_\ga(\tau, \partial_\tau)
\equiv
  j_\ga^k(\tau)
  \;\pmod{(\rho)}
\end{gather*}
that satisfies
\begin{gather*}
  \cJ_{\ga\delta}(\tau, \partial_\tau)
=
  \cJ_\ga(\delta\tau, \partial_{\delta\tau})\, \cJ_\delta(\tau)
\tx{.}
\end{gather*}
Note that as in the case of first order deformations, we suppress the dependency on~$k$.

Paralleling the case of first order deformations, we restrict our attention to~$\cJ_\ga(\tau, \partial_\tau)$ that we can write as
\begin{gather*}
  \cJ_\ga(\tau, \partial_\tau)
=
  j_\ga^k(\tau)\,
  A_\ga(\tau, \partial_\tau)
\end{gather*}
for some linear differential operator
\begin{gather*}
  A_\ga(\tau, \partial_\tau) \defcol \cO \lra \cO\llbrkt \rho_1,\ldots,\rho_n\rrbrkt
\quad
  \tx{with\ }
  A_\ga
\equiv
  1
  \;\pmod{(\rho)}
\tx{.}
\end{gather*}

The above conditions on~$\cJ_\ga(\tau)$ lead to the following constraints on~$A_\ga$:
\begin{gather}
\label{eq:formal_deformation_cocycle_line_bundle_relation}
  A_{\ga\delta}(\tau, \partial_\tau)
=
  \big( j_\delta^{-k}(\tau) A_\ga(\delta\tau, \partial_{\delta\tau}) j_\delta^k(\tau) \big)
  \cdot
  A_\delta(\tau, \partial_\tau)
\quad
  \tx{for all\ }\ga, \delta \in \Ga\tx{,}
\end{gather}
and, as already stated,~$A_\ga =1$ modulo~$(\rho) = (\rho_1,\ldots,\rho_n)$. 

The first of these tells us that~$A_\ga$ defines a non-abelian~$1$-cocycle in a certain ~$\Ga$\nbd{}module, while the second tells us that~$A_\ga$ has a multiplicative inverse as a formal power series with coefficients in linear maps. Thus, we find that formal deformations are encoded in certain non-abelian~$1$-cocycles. However, the space of all such linear maps is too large for our purposes. We wish to restrict further to sufficiently ``nice'' spaces, so that we recover the theory of first order deformations modulo the ideal~$(\rho)^2 = (\rho_1,\ldots,\rho_n)^2$.

To obtain such spaces, we observe that we can equip both 
\begin{gather}
\label{defeq:pronil}
  \frako_k^\rho \defeq (\rho)\, \frako_k\llbrkt \rho_1,\ldots,\rho_n\rrbrkt
\quad\tx{and}\quad
  \frakd_k^\rho \defeq (\rho)\, \frakd_k\llbrkt \rho_1,\ldots,\rho_n\rrbrkt
\end{gather}
with the structure of pro-nilpotent Lie algebras with~$\Ga$\nbd{}action. Hence, we can define their exponential groups
\begin{gather}
\label{defeq:groups_od}
  \bbO_k
\defeq
  \exp\big(\frako_k^\rho\big)
\subset
  \cO[\partial_\tau]\llbrkt \rho_1,\ldots,\rho_n\rrbrkt^\ast
\quad\tx{and}\quad
  \bbD_k
\defeq
  \exp\big(\frakd_k^\rho\big)
\subset
  \bbC[\tau,\partial_\tau]\llbrkt \rho_1,\ldots,\rho_n\rrbrkt^\ast
\end{gather}
as subgroups of the multiplicative group of invertible formal power series with coefficients in linear differential operators, which we annotate here and throughout by the superscript~$\ast$. We have associated projection maps
\begin{gather}
\label{eq:projection_group_od_to_germ}
  \bbO_k
\lthra
  \frako_k^\rho \big\slash (\rho)^2
\cong
  \bigoplus_{i = 1}^n \rho_i \frako_k
\quad\tx{and}\quad
  \bbD_k
\lthra
  \frakd_k^\rho \big\slash (\rho)^2
\cong
  \bigoplus_{i = 1}^n \rho_i \frakd_k
\tx{.}
\end{gather}

Furthermore, both~$\bbO_k$ and~$\bbD_k$ inherit from~\eqref{eq:def:double_slash_action} a weight~$k$~$\Ga$\nbd{}action given by
\begin{gather*}
  A(\tau,\partial_\tau)\big\|_k \ga
\defeq
  j_\ga(\tau)^{-k}\, A(\ga\tau,\partial_{\ga\tau})\, j_\ga^k(\tau)
\tx{.}
\end{gather*}
It is these (pro-unipotent) groups with~$\Ga$\nbd{}action that we will consider to define deformations of modular forms.

We make a final observation, which we apply in Section~\ref{ssec:calculation_of_second_order_deformations}, before returning to our discussion of deformations of modular forms.

\begin{proposition}
\label{prop:weight-indep}
There exists an isomorphism
\begin{gather*}
  \phi_k \defcol
  \frako_0 \lra \frako_k
\tx{,}\ 
  p\partial_\tau + q \lmto p\partial_\tau + \tfrac{k}{2}p^\prime + q
\tx{,}
\quad\tx{and by extension\ }
  \phi_k \defcol \frako_0^\rho \cong \frako_k^\rho
\tx{,}
\end{gather*}
of Lie algebras for every~$k$, compatible with the~$\Ga$\nbd{}action. Furthermore,~$\phi_k$ restricts to an isomorphism
\begin{gather*}
  \phi_k \big|_{\frakd_0^\rho} \defcol \frakd_0^\rho\cong\frakd_k^\rho
\tx{.}
\end{gather*}
\end{proposition}

\begin{proof}
We claim that the~$\rho$\nbd{}linear map induced by
\begin{gather*}
  \phi_{k}\defcol \frako_0 \lra \frako_k
\tx{,}\quad
  p\partial_\tau + q \lmto p\partial_\tau + \tfrac{k}{2}p^\prime + q
\end{gather*}
is an isomorphism of Lie algebras with~$\Ga$\nbd{}action. That it is an isomorphism of vector spaces is clear, as is the restriction to~$\frakd_0\to\frakd_k$. As such, it suffices to show that it is a homomorphism

We first show that~$\phi_k$ is compatible with the~$\Ga$\nbd{}action. Note that
\begin{align*}
   \phi_k\big( (a\partial_\tau+ b)\big\|_0\gamma \big)
&= \phi_k\big( j_\ga^2(\tau)a(\gamma\tau)\partial_\tau+b(\ga\tau)\big) \\
&= j_\ga^2(\tau)a(\ga\tau)\partial_\tau + kj^\prime_\ga(\tau)j_\ga(\tau)a(\ga\tau) + \tfrac{k}{2}a^\prime(\ga\tau)+b(\ga\tau) \\
&= j_\ga^{-k}(\tau)a(\ga\tau)\partial_{\ga\tau} j_\ga^k(\tau) + \tfrac{k}{2}a^\prime(\ga\tau)+b(\ga\tau)\\
&= \phi_k\big(a\partial_\tau+ b\big)\big\|_k\gamma.
\end{align*}
We then check that
\begin{align*}
   \big[ \phi_k(p\partial_\tau + r),\, \phi_k(q\partial_\tau + s) \big]
&= \big[ p\partial_\tau + \tfrac{k}{2}p^\prime + r,\, q\partial_\tau + \tfrac{k}{2}q^\prime + s \big] \\
&= \big[ p\partial_\tau,\, q\partial_\tau \big]
+ \tfrac{k}{2} \big(pq^{\prime\prime} -p^{\prime\prime}q\big) + ps^\prime - qr^\prime\\
&= \big(pq^\prime - p^\prime q\big)\partial_\tau
+ \tfrac{k}{2} \big(p q^\prime - p^\prime q \big)^\prime + ps^\prime - qr^\prime \\
&= \phi_k\big(\big[ p\partial_\tau+ r,\, q\partial_\tau+ s \big]\big)
\tx{.}
\end{align*}

Thus~$\phi_k$ is a Lie algebra isomorphism compatible with the~$\Ga$\nbd{}actions.
\end{proof}

\begin{corollary}\label{cor:weight-indep}
The isomorphism~$\phi_k$ induces isomorphisms
\begin{gather*}
  \Phi_k \defcol \bbO_0\cong \bbO_k
\quad\text{and}\quad
  \Phi_k \big|_{\bbD_0} \bbD_0 \cong\bbD_k
\end{gather*}
compatible with the~$\Ga$\nbd{}action.
\end{corollary}

\begin{proof}
This is an immediate consequence of the Lie group-Lie algebra correspondence. 
\end{proof}

\subsection{Formal deformations of modular forms}%
\label{ssec:formal_deformations_modular_forms}

The setup of formal deformations parallels the one of first order deformations. In particular, we restrict to formal deformation cocycles with algebraic coefficients. In Proposition~\ref{prop:equivalent-cocycles}, we examine more closely also the relation to deformations of line bundle. The key result of this section is the uniqueness stated in Theorem~\ref{thm:uniqueness_of_coboundary}.

Based on the above discussion of formal deformations of line bundles and the cocycle relation in~\eqref{eq:formal_deformation_cocycle_line_bundle_relation}, we make the following definition.

\begin{definition}
\label{def:algebraic_deformation_cocycle}
A formal algebraic deformation cocycle of weight~$k$ is an element
\begin{gather*}
  A_\bullet \in \rmZ^1(\Ga,\bbD_k)
\tx{.}
\end{gather*}
That is, it is a map
\begin{gather*}
  A_\bullet \defcol \Ga \lra \bbD_k
\end{gather*}
to the space of invertible power series such that
\begin{gather*}
  A_{\ga\delta} = A_\ga\big\|_k \delta\,\cdot A_\delta
\quad
  \tx{for all\ }\ga, \delta \in \Ga
\tx{.}
\end{gather*}

Further, we say two deformation cocycles~$A$,~$\wtd{A}$ are equivalent if they define the same cohomology class in the non-abelian group cohomology~$\rmH^1(\Ga,\bbD_k)$, which is to say there exists some~$C\in\bbD_k$ such that
\begin{gather*}
  \wtd{A}_\ga = C\big\|_k\ga\,\cdot A_\ga\cdot C^{-1}
\tx{.}
\end{gather*}
We call the image of a formal algebraic deformation cocycle in cohomology a deformation class.
\end{definition}

As in the case of first order deformations, we will refer to cocycles in Definition~\ref{def:first_order_deformation_cocycle} as formal deformation cocycles, and suppress the reference to their algebraic coefficients.

The next statement asserts that deformation cocycles for different weights can be related via the isomorphism in Corollary~\ref{cor:weight-indep}. As such, it suffices in the remainder of this work to consider only the weight~$0$ theory, and we will often suppress weight from our notation where the distinction is not needed.

\begin{lemma}
\label{lem:weight-independence}
For any~$k$, there exists an isomorphism of non-abelian cohomology sets
\begin{gather*}
  \rmH^1(\Ga,\bbD_0) \cong \rmH^1(\Ga,\bbD_k)
\end{gather*}
induced by the isomorphism of Corollary \ref{cor:weight-indep}.
\end{lemma}

\begin{proof}
Recall the isomorphism
\begin{gather*}
  \Phi_k \defcol \bbO_0 \lra \bbO_k
\end{gather*}
of Corollary \ref{cor:weight-indep}. We will use the same notation for its restriction to~$\bbD_0$. Suppose~$A_\bullet\in \rmZ^1(\Ga,\bbD_0)$, which is to say
\begin{gather*}A_{\ga\delta} = A_\ga\big\|_0 \delta\,\cdot A_\delta\end{gather*}
for all~$\ga$,~$\delta\in \Ga$. As~$\Phi_k$ is an isomorphism of groups, compatible with the~$\Ga$\nbd{}action, we have that
\begin{gather*}\Phi_k(A_{\ga \delta}) = \Phi_k(A_\ga)\big\|_k \delta\,\cdot \Phi_k(A_\delta)\end{gather*}
and hence 
\begin{gather*}\Phi_k(A_\bullet) \in \rmZ^1(\Ga,\bbD_k).\end{gather*}
As~$\Phi_k$ is an isomorphism, we must have
\begin{gather*}\rmZ^1(\Ga,\bbD_0)\cong \rmZ^1(\Ga,\bbD_k)\tx{.}\end{gather*}
To see that this descends to an isomorphism of cohomology groups, it is a simple computation to check that~$A_\bullet$ and~$\wtd{A}_\bullet$ define the same cohomology class in~$\rmH^1(\Ga,\bbD_0)$ if and only if~$\Phi_k(A_\bullet)$ and~$\Phi_k(\wtd{A}_\bullet)$ define the same cohomology class in~$\rmH^1(\Ga,\bbD_k)$.
\end{proof}

\begin{remark}\label{subsec:base-def}
Recall that formal deformation cocycles of weight~$k$ match the cocycles attached to formal deformations of the line bundle~$L_k$ in the sense of Definition~\ref{def:formal_deformation_line_bundle} Further, the two direct summands in the definition of~$\frako_k$ in~\eqref{eq:def:frako} correspond to deformations of the base space and the line bundle in that definition.

A deformation of the base space~$\cY_\Ga$ is uniquely determined by a deformation of~$\cO_{Y_\Ga}$, i.e.\@ a deformation of the sheaf corresponding to weight~$0$ modular functions. To avoid possible confusion, we note that in Theorem~\ref{thm:first-order-class} such a deformation of~$\cO_{Y_\Ga}$ is implemented by a cocycle with values in the tangent bundle. Given a weight~$k$ deformation cocycle
\begin{gather*}
  A_\bullet \in \rmZ^1(\Ga,\bbD_k)
\tx{,}
\end{gather*}
the weight~$0$ cocycle~$\Phi_k^{-1}(A_\bullet)$ defines a deformation of weight~$0$ modular forms, and hence a deformation~$\cY_\Ga$ of~$Y_\Ga$.

Geometrically, we may then think of the inverse isomorphism~$\Phi_k^{-1}$ as recovering the deformation~$\cY_\Ga$ of the base space~$Y_\Ga$ from the deformation of the line bundle~$L_k$, analogous to the projection~$\cE_{L_k} \to T_{Y_\Ga}$ that appears in Theorem~\ref{thm:first-order-class}.
\end{remark}

We next determine the exact relation between deformation cocycles and deformations of the line bundle~$L_k$.

\begin{proposition}%
\label{prop:equivalent-cocycles}
Equivalent formal deformation cocycles define isomorphic deformations of~\/$L_k$.
\end{proposition}

\begin{proof}
By Lemma \ref{lem:weight-independence}, it suffices to consider only the weight~$0$ case. Suppose~$\cL$ and~$\wtd{\cL}$ are deformations of~$L_0$, with corresponding deformation cocycles~$A_\ga$ and~$\wtd{A}_\ga$. Suppose~$A_\ga$ and~$\wtd{A}_\ga$ are equivalent, so that there exists~$C\in\bbD_0$ such that
\begin{gather*}\wtd{A}_\ga = C\big\|_k\ga\,\cdot A_\ga\cdot C^{-1}\tx{.}\end{gather*}
In order to show~$\cL\cong \wtd{\cL}$, it suffices to construct compatible isomorphisms between
\begin{gather*}
  \cL(U)
=
  \big\{
  f \defcol \pi^{-1}(U) \to \bbC\llbrkt \rho_1,\ldots,\rho_n\rrbrkt
  \condcol
  j_\ga^{-k}(\tau)f(\ga\tau) = A_\ga(\tau,\partial_\tau)f(\tau)
  \big\}
\end{gather*}
and
\begin{gather*}
  \wtd{\cL}(U)
=
  \big\{
  f \defcol \pi^{-1}(U) \to \bbC\llbrkt \rho_1,\ldots,\rho_n\rrbrkt
  \condcol
  j_\ga^{-k}(\tau)f(\ga\tau) = \wtd{A}_\ga(\tau,\partial_\tau)f(\tau)
  \big\}
\end{gather*}
for every open~$U\subset \cY_\Ga$. Here,~$\pi:\HS\to\cY_\Ga$ is the standard quotient map.

We claim that multiplication by~$C$ defines an isomorphism for every open~$U$. It is clearly invertible and uniform across all opens, so it just remains to check that it is well defined. Suppose~$f\in\cL(U)$. Then 
\begin{gather*} Cf:\pi^{-1}(U)\to\bbC\llbrkt \rho_1,\ldots,\rho_n\rrbrkt\end{gather*}
and
\begin{align*}
j_\ga^{-k}(\tau)C(\ga\tau,\partial_{\ga\tau}) f(\ga\tau) &= j_{\ga}^{-k}(\tau)C(\ga\tau,\partial_{\ga\tau})j_\ga^{k}(\tau) j_\ga^{-k}(\tau) f(\ga\tau)\\
&= C\big\|_k\ga\, \cdot A_\ga f(\tau)\\
&= C\big\|_k\ga\, \cdot A_\ga \cdot C^{-1} Cf(\tau)\\
&= \wtd{A}_\ga \cdot Cf(\tau)\tx{,}
\end{align*}
and so~$Cf\in \wtd{\cL}(U)$. Thus~$\cL \cong \wtd{\cL}$.
\end{proof}

We note one useful property of our deformation cocycles. For the expansion in the variables~$\rho_i$ we write
\begin{gather}
\label{eq:deformation_cocyle_coefficients}
  A_\ga(\tau, \partial_\tau)
=
  \sum_{m \in \ZZ^n}
  a_{\ga,m}(\tau, \partial_\tau)\,
  \rho_1^{m_1} \cdots \rho_n^{m_n}
\tx{.}
\end{gather}

Note that, from the definition of~$\bbD_k$ as an exponential in~\eqref{defeq:groups_od}, we must have that for all~$\ga \in \Ga$ and all~$m \in \ZZ^n$ we have~$\ord(a_{\ga,m}) \le |m|$. It is also worth noting that, in weight 0, these deformation cocycles act as homomorphisms:
\begin{gather*}
  A_\ga(\tau,\partial_\tau)(fg) = A_\ga(\tau,\partial_\tau)(f)\, A_\ga(\tau,\partial_\tau)(g)
\tx{.}
\end{gather*}
For a power series with constant term~$1$, this is equivalent to~$\log(A_\ga)$ being a derivation, i.e.
\begin{gather*}
  \log(A_\ga) \in \frakd_0^\rho
=
  (\rho_1,\ldots,\rho_n)\,
  \frakd_0\llbrkt \rho_1,\ldots,\rho_n\rrbrkt
\tx{,}
\end{gather*}
which is true, as~$\bbD_0$ is defined as the exponential of~$\frakd_0^\rho$. We will later extend this to all weights to ensure that we can construct a ``universal'' deformation cocycle, acting as a homomorphism on the ring of modular forms. 

We now wish to construct sections of the deformed line bundle. More precisely, we wish to construct a morphism of sheaves
\begin{gather*}\cO\to \cO \llbrkt \rho_1,\ldots,\rho_n\rrbrkt\end{gather*}
such that the restriction to the space of modular forms of weight~$k$ induces a map
\begin{gather*}\rmH^0(Y_\Ga,L_k) \lra \rmH^0(\cY_\Ga,\cL_k).\end{gather*}
We will call such a map a deformation of modular forms of weight~$k$, as defined below.

\begin{definition}
\label{def:algebraic_deformation_section}
A deformation section of weight~$k$ and deformation cocycle~$A_\bullet$ is
a trivialisation~$B(\tau, \partial_\tau)$ under the coboundary map
\begin{align*}
  \rmZ^0\big( \Ga,\, \bbO_k \big)
&
\lra{}
  \rmZ^1\big( \Ga,\, \bbO_k \big)
\supseteq
  \rmZ^1\big( \Ga,\, \bbD_k \big)
\\
  B
&\lmto
  A_\bullet
\tx{.}
\end{align*}
That is, it is an invertible power series
\begin{gather*}
  B(\tau,\partial_\tau) \in \bbO_k
\quad\tx{such that\ }
  B\big\|_k\gamma\,\cdot B^{-1}
=
  A_\ga
\tx{.}
\end{gather*}
\end{definition}

\begin{remark}
\label{rm:algebraic_deformation_section_existence}
Given a deformation cocycle 
\begin{gather*}A_\bullet\in \rmZ^1(\Ga,\bbD_k)\end{gather*}
the existence of a deformation section~$B$ of type~$A_\bullet$ is equivalent to~$A_\bullet$ being equivalent to the trivial deformation~$A_\bullet=1$ over~$\bbO_k$. That is to say that~$A_\bullet$ is a~$1$-coboundary in~$\rmZ^1(\Ga,\bbO_k)$. In particular, a necessary and sufficient condition for the existence of~$B$ is that the cohomology class of~$A_\bullet$ vanishes under the natural map
\begin{gather*}\rmH^1(\Ga,\bbD_k) \lra \rmH^1(\Ga,\bbO_k)\tx{.}\end{gather*}
\end{remark}

Paralleling the situation of first order deformations, a deformation section~$B$ yields a deformation of a modular form~$f$ by its action~$B(\tau, \partial_\tau)\, f$ as a differential operator. The following is then an immediate consequence of the definition.

\begin{lemma}
\label{lem:deformed-transformation}
Let~$A_\bullet$ be a deformation cocycle of weight~$k$, and~$B$ be a deformation section of type~$A_\bullet$. For a modular form~$f$ of weight~$k$, the deformed modular form transforms as
\begin{gather*}
  \big( B(\tau, \partial_\tau)(f) \big)(\ga\tau)
=
  j_\gamma^{k}(\tau)\, A_\ga(\tau, \partial_\tau)\, \big( B(\tau, \partial_\tau) f \big)(\tau)
\tx{.}
\end{gather*}
\end{lemma}

While existence of a deformation section of a given type requires more work as explained in Remark~\ref{rm:algebraic_deformation_section_existence}, we can easily show that, given a deformation cocycle~$A_\bullet$, there exists at most one deformation section of type~$A_\bullet$.

\begin{theorem}%
\label{thm:uniqueness_of_coboundary}
Given a deformation class
\begin{gather*}
  A_\bullet
\in
  \ker\Big(
  \rmH^1(\Ga,\bbD_k)
  \lra
  \rmH^1(\Ga,\bbO_k)
  \Big)
\end{gather*}
represented by~$A_\bullet \in \rmZ^1(\Ga, \bbD_k)$, the corresponding~$\bbO_k$\nbd{}coboundary is unique:
\begin{gather*}
  \exists !\,
  B(\tau,\partial_\tau) \in \bbO_k
\tx{\ such that}\quad
  A_\ga
=
  B \big\| \ga \cdot B^{-1}
\tx{.}
\end{gather*}
\end{theorem}

\begin{proof}
Using the isomorphisms of Corollary \ref{cor:weight-indep}, it suffices to establish this in weight~$0$. Suppose we have two power series~$B,\,C\in\bbO_0$ such that
\begin{gather*}B\big\|_k\ga\,\cdot B^{-1}= A_\ga = C\big\|_k\ga\,\cdot C^{-1}.\end{gather*}
This implies that
\begin{gather*} (C^{-1}\cdot B)\big\|_k\ga = C^{-1}\cdot B\end{gather*}
is an invariant function. Call it~$X$ and write
\begin{gather*}X(\tau) = \sum_{m\in \ZZ^n}\sum_{k=0}^{k_m} x_{m,k}(\tau)\partial_\tau^k \rho_1^{m_1}\cdots \rho_n^{m_n}\end{gather*}
Considering the coefficient of the highest order derivative for each~$m$, invariance of~$X$ implies that
\begin{gather*}j_\ga^{2k_m}(\tau)x_{m,k_m}(\ga\tau)=x_{m,k_m}(\tau)\end{gather*}
and hence~$x_{m,k_m}(\tau)$ is a modular form of weight~$-2k_m$. There are no  non-zero modular forms of negative weight, and hence~$x_{m,k_m}(\tau)=0$ unless~$k_m=0$, where~$x_{m,0}(\tau)$ is a modular form of weight~$0$. Hence
\begin{gather*}X(\tau) =\sum_{m\in \ZZ^n}x_{m}(\tau)\rho_1^{m_1}\cdots \rho_n^{m_n}\end{gather*}
with~$\{x_{m}(\tau)\}$ a collection of weight 0 modular forms.

However, as both~$B$ and~$C$ are elements of the group~$\bbO_0$, so is~$X=C^{-1}B$. From the definition of~$\bbO_0$ as the exponential of~$\frako_0$, it is clear to see that there is no non-zero element~$x$ of~$\frako_0$ such that the coefficients of~$\exp(x)$ are differential operators of degree~$0$. In particular, there is no non-zero~$x$ such that~$X=\exp(x)$. Thus, we must have~$X=1$.
\end{proof}

\subsection{Extension of first order deformations}%
\label{ssec:extension_first_order_deformations}

The goal of this section is to complement Theorem~\ref{thm:uniqueness_of_coboundary} and show that first order deformation cocycles extend to formal ones.

\begin{remark}
It is as yet unclear whether Bogo's construction yields an algebraic prolongation of our common first order deformation. By direct computation, one may show that Bogo's construction agrees with ours up to second order (cf.\@ Section~\ref{ssec:calculation_of_second_order_deformations}). The recursive nature of Bogo's construction would suggest his deformation is algebraic in the sense of this section, but it is as yet unclear how to prove this.
\end{remark}

Let us first consider the first order approximations
\begin{gather*}
  A_\ga(\tau, \partial_\tau)
=
  1
  +
  \sum_{i=1}^n
  a_{\ga,i}(\tau, \partial_\tau) \rho_i
  +
  O\big( \rho^2 \big)
\tx{,}
\end{gather*}
where we use~$O(\rho^2)$ to refer to any terms of order at least~$2$ in~$\rho_1,\ldots,\rho_n$.

The non-abelian cocycle condition
\begin{gather*}A_{\ga\delta} = A_{\ga}\big\|_k\delta\,\cdot A_\delta\end{gather*}
reduces to the linear cocycle condition
\begin{gather*}
  a_{\ga\delta,i} = a_{\ga,i}\big\|_k\delta + a_{\delta,i}
  \quad\text{ for all }1\leq i \leq n
\tx{.}
\end{gather*}
From the results of Section~\ref{sec:first-order-deformations}, specifically Proposition~\ref{prop:first-order-cusp-forms} we know that solutions to this cocycle condition are constructed from weight 4 modular forms.

Given an algebraic deformation to first order, the natural question to ask is whether it can be prolonged to a formal deformation. As~$\Ga$ has virtual cohomological dimension~$2$, this is always possible.

\begin{theorem}%
\label{thm:extension-of-first-order}
The natural projection map
\begin{gather*}
  \rmZ^1(\Ga,\bbD_k)
\lra
  \bigoplus_{i=1}^n \rmZ^1(\Ga,\frakd_k)
\end{gather*}
associated with the one in~\eqref{eq:projection_group_od_to_germ} is surjective. That is, given~$a_{\bullet,i}\in \rmZ^1(\Gamma, \frakd_k)$ for~$i=1,\ldots,n$, it is always possible to find~$A_\bullet\in \rmZ^1(\Gamma,\bbD_k)$ such that 
\begin{gather*}
  A_\bullet = 1 + \sum_{i=1}^n a_{\bullet,i} \rho_i + O\big(\rho^2\big)
\tx{.}
\end{gather*}

In other words, given a first order approximation of a deformation cocycle of weight~$k$, it is always possible to prolong it to a formal deformation cocycle.
\end{theorem}

\begin{proof}
Denote by~$\bbD_k\big\slash (\rho)^N$ the image of the natural map
\begin{gather*}
	\bbD_k
\lra
	\bbC[\tau,\partial_\tau]\llbrkt \rho_1,\ldots,\rho_n\rrbrkt \big\slash (\rho)^N
\tx{,}
\end{gather*}
and note that every element of~$\bbD_k$ is uniquely determined by its image in~$\big\{\bbD_k\big\slash (\rho)^N\big\}_{N\geq 1}$. That is to say that~$\bbD_k$ is the projective limit of the system~$\big\{\bbD_k\big\slash (\rho)^N\big\}_{N\geq 1}$ with the obvious surjections. As such, it suffices to show that the natural projection
\begin{gather*}
  \rmZ^1\big( \Ga,\, \bbD_k \big\slash (\rho)^{N+1} \big)
\lra
  \rmZ^1\big( \Ga,\, \bbD_k \big\slash (\rho)^{N} \big)
\end{gather*}
is surjective for every~$N\geq 1$. In fact, it suffices to show the corresponding fact for the cohomology groups. For every~$N \geq 1$, we have a short exact sequence of~$\Ga$\nbd{}groups
\begin{gather*}
  0
\lra
  (\rho)^N\, \frakd_k
\lra
  \bbD_k \big\slash (\rho)^{N+1}
\lra
  \bbD_k\big\slash (\rho)^{N}
\lra
  0
\tx{.}
\end{gather*}
As~$\frakd_k$ is an abelian group in this context, this implies that we have an exact sequence of pointed sets
\begin{multline*}
  0
\lra
  (\rho)^N\, \big(\frakd_k\big)^\Ga
\lra
  \big( \bbD_k \big\slash (\rho)^{N+1} \big)^\Ga
\lra
  \big( \bbD_k \big\slash (\rho)^{N} \big)^\Ga
\lra \cdots
\\
\cdots \lra
  \rmH^1\big(\Ga, \frakd_k\big)
\lra
  \rmH^1\big(\Ga, \bbD_k\big\slash (\rho)^{N+1} \big)
\lra
  \rmH^1\big(\Ga, \bbD_k\big\slash (\rho)^{N} \big)
\lra
  \rmH^2\big(\Ga, \frakd_k\big)
\tx{.}
\end{multline*}
As~$\rmH^2(\Ga,V) = 0$ for any finite dimensional representation~$V$, the final term in this sequence vanishes, and hence
\begin{gather*}
	\rmH^1\big(\Ga, \bbD_k\big\slash (\rho)^{N+1}\big) 
\lra 
	\rmH^1\big(\Ga, \bbD_k\big\slash (\rho)^{N}\big)
\end{gather*}
is surjective.
\end{proof}

\subsection{Calculation of second order deformations}%
\label{ssec:calculation_of_second_order_deformations}

We have shown that we can extend any first order deformation cocycle to a formal algebraic deformation cocycle. However, the arguments of Proposition~\ref{thm:extension-of-first-order} are not constructive, and so we will here illustrate a method for explicit computation. Specifically, we will compute an example of a deformation cocycle to second order.

In order to do this, we consider the logarithmic perspective on our cocycle conditions. We do so using (a special case of) the Baker-Campbell-Hausdorff formula~\cite{dynkin47,bourbaki72}.

\begin{theorem}
Let~$X,\,Y\in\frakg$ be elements of a pro-nilpotent Lie algebra. Then there exists~$Z\in\frakg$ such that~$\exp(X) \exp(Y) = \exp(Z)$, given by the Baker-Campbell-Hausdorff (BCH) formula
\begin{gather}\label{eq:bch}
  Z=\BCH(X,Y)=X + Y +\tfrac{1}{2}[X,Y]+\cdots
\end{gather}
\end{theorem}

\begin{remark}
For general Lie algebras and Lie groups, convergence of the Baker-Campbell-Hausdorff formula is not assured. However, if one has some meaningful notion of convergent series of elements of the Lie algebra, the above proposition holds for sufficiently ``small''~$X$ and~$Y$.
\end{remark}

We will be considering Lie subalgebras of~$\frako_k^\rho$, as defined in Section \ref{sec:first-order-deformations}. This is a pro-nilpotent Lie algebra, and the coefficient of each~$\rho_1^{m_1}\ldots \rho_n^{m_n}$ in the BCH formula will only involve finitely many terms. As such, we may freely apply the BCH formula to our setting.

Suppose we have a deformation cocycle~$A_\bullet = \exp(a_\bullet)$ of weight~$k$ for some~$a_\bullet\in\frakd_k^\rho$.  Then the cocycle condition is equivalent to
\begin{gather}\label{eq:log-cocycle-condition}
  a_{\ga\delta} = \BCH\big( a_\ga\big\|_k\delta,a_\delta \big)
\tx{.}
\end{gather}
The existence of a deformation section is then equivalent to the existence of a~$b\in\frako_k^\rho$ such that
\begin{gather}\label{eq:log-coboundary-condition}
  a_\ga = \BCH(b\big\|_k\ga,-b)\tx{.}
\end{gather}
By considering the coefficients of each~$\rho_1^{m_1}\ldots\rho_n^{m_n}$ in Equations~\eqref{eq:log-cocycle-condition} and~\eqref{eq:log-coboundary-condition}, we obtain a system of inhomogeneous cohomological equations, which we can solve recursively. 

\begin{definition}\label{def:log-deformation-data}
An element~$a_\bullet\in\frakd_k^\rho$ satisfying Equation~\eqref{eq:log-cocycle-condition} is called a logarithimic cocycle of weight~$k$. Given a logarithmic cocycle~$a_\bullet$, an element~$b\in\frako_k^\rho$ satisfying Equation~\eqref{eq:log-coboundary-condition} is called a logarithmic deformation section of type~$a_\bullet$.
\end{definition}

\begin{example}
For~$n=1$, and~$a_\ga = a^{(1)}_\ga\rho + a^{(2)}_\ga\rho^2+\cdots$, the first two pairs of equations for a deformation cocycle and deformation section are
\begin{gather}\label{eq:second-order-log}
\begin{alignedat}{2}
  a^{(1)}_{\ga\delta} &= a^{(1)}_\ga\big\|_k\delta + a^{(1)}_\delta
\tx{,}\qquad&
  a^{(2)}_{\ga\delta} &= a^{(2)}_\ga\big\|_k\delta + a^{(2)}_\delta + \tfrac{1}{2} \big[a^{(1)}_\ga\big\|_k\delta,a^{(1)}_\delta \big]
\tx{,}\\
  a^{(1)}_\ga &= b^{(1)}\big\|_k\ga - b^{(1)}
\tx{,}\qquad&
  a^{(2)}_\ga &= b^{(2)}\big\|_k\ga - b^{(2)} - \tfrac{1}{2} \big[ b^{(1)}\big\|_k\ga,b^{(1)} \big]
\tx{.}
\end{alignedat}
\end{gather}
\end{example}

The solvability of the cocycle equations is guaranteed, as in Proposition \ref{thm:extension-of-first-order}, by~$\Ga$ having virtual cohomological dimension 1. The equations for the section are more interesting. Via the isomorphisms of Proposition \ref{prop:weight-indep}, it suffices to consider the weight~$0$ case. 

From the discussions of Section \ref{sec:first-order-deformations}, a first order solution is given by
\begin{gather*}
  \big( a^{(1)}_\bullet,b^{(1)} \big)
=
  \big( 2p_{h,\bullet}\partial_\tau,\, 2\tilde{h}\partial_\tau \big)
\tx{,}
\end{gather*}
where~$h$ is a weight~$4$ cusp form,~$\tilde{h}$ is the associated Eichler integral, and~$p_{h,\bullet}$ is the associated period polynomial.

Let
\begin{gather*}
  a^{(2)}_\bullet= p^{(2)}_\bullet \partial_\tau \in \frakd_0
\tx{,}\qquad
  b^{(2)}= f\partial_\tau\in \frako_0\tx{,}
\end{gather*}
for some quadratic polynomial~$p^{(2)}$ and some holomorphic function~$f$. After expanding the definition of the Lie bracket in~\eqref{eq:def:lie_bracket_ok}, the equation for the section is then equivalent to
\begin{gather*}
  p^{(2)}_\ga(\tau) = j_\ga^2(\tau)f(\ga\tau) - f(\tau) - 2\left( j_\ga^2(\tau)\tilde{h}(\ga\tau)\tilde{h}^\prime(\tau) - 2j_\ga^\prime(\tau)j_\ga(\tau)\tilde{h}(\ga\tau)\tilde{h}(\tau) - \tilde{h}^\prime(\ga\tau)\tilde{h}(\tau)\right)
\tx{.}
\end{gather*}
By definition of the period polynomial, we have that
\begin{gather*}j_\ga^2(\tau)\tilde{h}(\ga\tau) = \tilde{h}(\tau) + p_{h,\ga}(\tau)\tx{,}\end{gather*}
allowing us to simplify the quadratic term:
\begin{gather}\label{eq:quadratic-log-term}
  p^{(2)}_\ga(\tau)
=
  j_\ga^2(\tau)f(\ga\tau) - f(\tau)
- 2\big(
  \big( \tilde{h}(\tau)+p_{h,\ga}(\tau) \big)\tilde{h}^\prime(\tau)
  - \big(\tilde{h}^\prime(\tau) + p_{h,\ga}^\prime(\tau)\big) \tilde{h}(\tau)
  \big)
\tx{.}
\end{gather}
As the left hand side is a quadratic polynomial, it is annihilated by taking the third derivative and we therefore wish to find a function~$f$ such that
\begin{gather*}
  j_\ga^{-4}(\tau) \frac{\rmd^3f}{\dtau^3}(\ga\tau) - \frac{\rmd^3f}{\dtau^3}
=
  -2h^\prime(\tau)p_{h,\ga}(\tau) - 4h(\tau)p^\prime_{h,\ga}(\tau)
\tx{,}
\end{gather*}
where we have used Bol's identity in the first term and on the right hand side the fact that
\begin{gather*}
  \frac{\rmd^3\tilde{h}}{\dtau^3}(\tau) = h(\tau)
\tx{.}
\end{gather*} 
Using the modular properties of the cusp form~$h$, we can show that the right hand side is equal to
\begin{gather*}
  -(c\tau+d)^{-4}\,
  \big( 2h^\prime(\gamma\tau)\tilde{h}(\gamma\tau) + 4h(\gamma\tau)\tilde{h}^\prime(\gamma\tau) \big)
  + \big( 2h^\prime(\tau)\tilde{h}(\tau) + 4 h(\tau)\tilde{h}^\prime(\tau) \big)
\tx{.}
\end{gather*}
A particular solution to this is given by taking~$f$ such that
\begin{gather*}
  \frac{\rmd^3f}{\dtau^3}(\tau)
=
  -2h^\prime(\tau)\tilde{h}(\tau) + 4h(\tau)\tilde{h}^\prime(\tau)
\tx{,}
\end{gather*}
which we recognise as the Rankin-Cohen bracket~$-[h,\tilde{h}]_1$~\cite{bruinier-van-der-geer-harder-zagier-2008}. Such an~$f$ is given by
\begin{gather*}-\int_\tau^{i\infty} [h,\tilde{h}]_1(z)(\tau-z)^2 \,\dz\tx{.}\end{gather*}

\begin{lemma}\label{lem:second-order-solutions}
A solution to the cohomological equations  given in Equation~\eqref{eq:second-order-log} is given by
\begin{alignat*}{2}
  a^{(1)}_\bullet &= 2p_{h,\bullet}\partial_\tau
\tx{,}\qquad&
  a^{(2)}_\bullet &= p^{(2)}_\bullet \partial_\tau
\tx{,}\\
  b^{(1)} &= 2\tilde{h}\partial_\tau
\tx{,}\qquad&
  b^{(2)} &= f(\tau)\partial_\tau
\tx{,}
\end{alignat*}
where~$p^{(2)}$ is defined by Equation~\eqref{eq:quadratic-log-term}, and 
\begin{alignat*}{2}
  f(\tau)
={}
  && 12\tau^2 &\,\big(\Lambda_\tau(h,h;1,0) -  \Lambda_\tau(h,h;0;1)\big)\\
  && -4\tau &\,\big(7\Lambda_\tau(h,h;2,0) - 8\Lambda_\tau(h,h;1,1) +\Lambda_\tau(h,h;0,2)\big)\\
  && +4 &\,\big(4\Lambda_\tau(h,h;3,0) - 3\Lambda_\tau(h,h;2,1) - \Lambda_\tau(h,h;1,2)\big)
\end{alignat*}
is defined in terms of functional multiple modular values (Definition \ref{def:func_mult_mod}).
\end{lemma}
\begin{proof}
That this defines a solution is an immediate consequence of the above discussion if we can show that
\begin{gather*}f(\tau) = -\int_\tau^{i\infty} [h,\tilde{h}]_1(z)(\tau-z)^2\,\dz\tx{.}\end{gather*}
Expanding the Rankin-Cohen bracket and integrating by parts, we have that
\begin{align*}
 -\int_\tau^{i\infty} [h,\tilde{h}]_1(z)(\tau-z)^2\,\dz &= -2\int_{\tau}^{i\infty} h^\prime(z)\tilde{h}(z)(\tau-z)^2\,\dz +4\int_{\tau}^{i\infty}h(z)\tilde{h}^\prime(z)(\tau-z)^2\,\dz\\
&= 6\int_{\tau}^{i\infty} h(z)\tilde{h}^\prime(z)(\tau-z)^2\,\dz -4 \int_{\tau}^{i\infty}h(z)\tilde{h}(z)(\tau-z)\,\dz\tx{.}
\end{align*}
Recall that, by definition
\begin{gather*}\tilde{h}(\tau) = \int_{\tau}^\infty h(z)(\tau-z)^2\,\dz = \tau^2\Lambda_\tau(h;0) - 2\tau\Lambda_\tau(h;1) + \Lambda_\tau(h;2)\end{gather*}
and similarly
\begin{gather*}\tilde{h}^\prime(\tau) = 2\int_\tau^\infty h(z)(\tau-z) \,\dx = 2\tau\Lambda_\tau(h,0) -2\Lambda_\tau(h,1)\tx{.}\end{gather*}
We therefore see that our integral is given in terms of multiple modular values as claimed.
\end{proof}

\begin{remark}
As the derivative of a functional multiple modular is again a functional multiple modular value, we see that the coefficients of~$p^{(2)}$ can be described in terms of classical multiple modular values, and are therefore periods in the sense of~\cite{kontsevich-zagier-2001}.
\end{remark}

\section{Universal and canonical deformations}%
\label{sec:universal_canonical}

The goal of this section is to examine the universal and canonical deformation associated with Brown's canonical cocycle, and its totally holomorphic quotient. The main construction of this deformation is presented in Section~\ref{ssec:canonical_deformation}, presented explicitly in the totally holomorphic case. We preface it with two more general discussions. In Section~\ref{ssec:universal_deformation} we first establish the existence and uniqueness of universal deformation families, that is families of deformations for all weights that behave as homomorphisms on the graded ring of modular forms. Section~\ref{ssec:existence_deformation_section} picks up the loose ends of Section~\ref{ssec:extension_first_order_deformations}, where we showed that first order deformation coccycles extend to formal ones and that associated deformation sections are unique if they exist. We show that there is a deformation section associated with any deformation cocycle.

\subsection{A universal deformation family}%
\label{ssec:universal_deformation}

Recall from Proposition~\ref{prop:weight-indep} that we have an isomorphism between~$\frako_k$ and~$\frako_0$. The corresponding isomorphism between~$\bbO_k$ and~$\bbO_0$ is stated in Corollary~\ref{cor:weight-indep}. The goal of this section is to show that it extends to isomorphisms between deformation cocycles and deformation sections of varying weight.

Given a family 
\begin{gather*}
  \big\{ \big(A^{(k)},B^{(k)} \big) \big\}_{k\geq 0}
\end{gather*}
of deformation cocycles ~$A^{(k)}$ of weight~$k$ and deformation sections~$B^{(k)}$ of type~$A^{(k)}$, we call the family universal if for every pair of modular forms~$f$ of weight~$k$ and~$g$ of weight~$\ell$
\begin{gather*}
  B^{(k+\ell)}(fg)
=
  B^{(k)}(f)\, B^{(\ell)}(g)
\tx{.}
\end{gather*}

If the graded ring of modular forms for~$\Ga$ is generated by its weight~$0$ and weight~$1$ components, then a universal family is uniquely determined by its weight~$0$ and~$1$ components --- this is precisely how Bogo defines his analytic deformations~\cite{bogo21}.

In the next statement we assert the existence of such a family given a pair~$(A^{(0)},B^{(0)})$ of a weight~$0$ deformation cocycle and deformation section. Note that this proposition guarantees a deformation of weight~$k$ given a deformation in weight~$0$, independent of the existence of non-trivial modular forms of weight~$k$. This provides a distinct contrast to Bogo's construction that explicitly relies on the existence of a weight~$0$ and weight~$1$ modular form.

\begin{proposition}\label{prop:universal_family}
Given a weight~$0$ deformation cocycle~$A_\bullet$ and a deformation section~$B$ of type~$A_\bullet$, the collection
\begin{gather*}
  \big\{ \big(\Phi_k(A_\bullet),\Phi_k(B) \big) \big\}_{k\geq 0}
\tx{,}
\end{gather*}
where the~$\Phi_k$ are the isomorphisms of Corollary~\ref{cor:weight-indep}, is a universal family.
\end{proposition}

\begin{proof}
The statement of the proposition consists of three elements
\begin{enumerateroman}
  \item the~$\Phi_k(A_\bullet)$ are cocycles for the~$\big\|_k$ action of~$\Ga$,
  \item~$\Phi_k(A_\ga) = \Phi_k(B)\big\|_k\gamma\,\cdot \Phi_k(B)^{-1}$ for all~$k\geq 0$,
  \item for every modular form~$f$ of weight~$k$ and~$g$ of weight~$\ell$, 
  \begin{gather*}\Phi_{k+\ell}(B)(fg) = \Phi_k(B)(f)\cdot \Phi_\ell(B)(g).\end{gather*}
\end{enumerateroman}
The first two of these follow from the~$\Ga$-equivariance of~$\Phi_k$, as described in the proof of Lemma \ref{lem:weight-independence}. To prove the third, as 
\begin{gather*}
  \exp(\phi_k)(b\partial_\tau)= \Phi_k\left(\exp(b\partial_\tau)\right)
\end{gather*}
it suffices to show that
\begin{gather*}
    \exp(\phi_{k+\ell})(b\partial_\tau)(fg) = \exp(\phi_k)(b\partial_\tau)(f)\cdot  \exp(\phi_\ell)(b\partial_\tau)(g)
\tx{.}
\end{gather*}
We define a linear map
\begin{gather*}
  \psi:\frako_0^\rho \lra \Der\big( \rmM,(\rho)\cO\llbrkt \rho_1,\ldots,\rho_n\rrbrkt \big)
\end{gather*}
from the Lie algebra~$\frako_0^\rho$ to the space of derivations from the graded ring of modular forms as follows. For a modular form~$f$ of weight~$k$, the action of~$b\partial_\tau\in\frako_0$ is given by
\begin{gather*}\psi(b\partial_\tau)(f):= \phi_k(b)(f) = 2bf^\prime + kb^\prime f\tx{.}\end{gather*}
Recall here that~$\phi_k:\frako_0\to \frako_k$ is the isomorphism of Lie algebras of Proposition \ref{prop:weight-indep}. To see that this is a well defined derivation of the graded ring, note that
\begin{gather}\label{eq:graded-derivation}
\begin{aligned}
  \psi(b\partial_\tau)(fg) &= \left(2b\partial_\tau + (k+\ell)b^\prime\right)(fg)\\
  &= \left(2b\partial_\tau + kb^\prime\right)(f)g + f\left(2b\partial_\tau + \ell b^\prime\right)(g)= \psi(b\partial_\tau)(f)g + f\psi(b\partial_\tau)(g)
\end{aligned}
\end{gather}
for modular forms~$f$ of weight~$k$ and~$g$ of weight~$\ell$. If it were possible to interpret~$\psi(b\partial_\tau)$ as a derivation from a graded space to itself, we would be done via exponentiation. We have to be a bit more careful, but will proceed analogously.

Note that, for any holomorphic functions~$f$ and~$g$, the central equality of Equation~\eqref{eq:graded-derivation} holds. As such, we can show by induction that
\begin{gather*}
  \big( 2b\partial_\tau+ (k+\ell)b^\prime \big)^N(fg)
=
  \sum_{r+s=N}\mbinom{N}{r} \big(2b\partial_\tau + kb^\prime\big)^r(f) \cdot \big(2b\partial_\tau+\ell b^\prime\big)^s(g)
\end{gather*}
for all~$N\geq 0$ and hence
\begin{align*}
&
  \exp(\phi_{k+\ell})(b\partial_\tau)\, (fg)
=
  \sum_{N\geq 0} \mfrac{1}{N!}\phi_{k+\ell}(b\partial_\tau)^N\, (fg) \\
={}&
  \sum_{r,s\geq 0}\mfrac{1}{r!s!} \phi_k(b\partial_\tau)^r(f)\cdot \phi_\ell(b\partial_\tau)^s(g)
=
  \exp(\phi_{k})(b\partial_\tau)(f)\cdot\exp(\phi_{\ell})(b\partial_\tau)(g)
\tx{.}
\end{align*}
\end{proof}

\begin{remark}
At very few points in any proof have we used the integrality of~$k$. Via our isomorphisms~$\Phi_k$, it is possible to formally define a deformation space of weight~$k$ for any complex~$k$. For any~$\alpha$ for which there exists a collection of holomorphic functions~$\{\frakj_{\ga,\alpha}(\tau)\}_{\ga\in\Ga}$ on the upper half plane such that
\begin{gather*}\frakj_{\ga\delta,\alpha}(\tau) = \frakj_{\ga,\alpha}(\delta\tau)\frakj_{\delta,\alpha}(\tau)\quad\text{and}\quad j_\ga^2(\tau)\frakj^\prime_{\ga,\alpha}(\tau) = \alpha j^\prime_\ga(\tau)j_\ga(\tau)\frakj_{\ga,\alpha}(\tau)\tx{,}\end{gather*}
this formal deformation space defines genuine deformation weight~$\alpha$ modular forms.

For example, for~$\Ga=\Ga_1(4)$, we can define modular forms of weight~$\frac{1}{2}$, with automorphy factor
\begin{gather*}
  \frac{\theta(\ga\tau)}{\theta(\tau)}
\quad\tx{or}\quad
  \frac{\eta(\ga\tau)}{\eta(\tau)}
\qquad\text{where\ }
  \theta(\tau) \defeq \sum_{n\in\bbZ} e^{2\pi i n^2 \tau}
\tx{,}\quad
  \eta(\tau) \defeq e^{2 \pi i \frac{1}{24} \tau} \prod_{n=1}^\infty \big( 1 - e^{2 \pi i n \tau} \big)
\tx{.}
\end{gather*}
It is easy to verify that these satisfy the above conditions, and so our theory extends to deformations of modular forms of half-integral weight.
\end{remark}

\subsection{Existence of deformation sections}%
\label{ssec:existence_deformation_section}

The goal of this section is to establish the existence of formal deformation sections associated to any given deformation cocycle. We achieve this in Theorem~\ref{thm:existence_and_uniqueness_of_section}. However, much like the proof of Theorem~\ref{thm:extension-of-first-order}, the proof of this is entirely non constructive. As such, we present a more explicit recipe in Proposition~\ref{prop:extending_sections} and Corollary~\ref{cor:constructing-sections} for constructing a logarithmic section, modulo some subtleties surrounding convergence of integrals.

\begin{theorem}%
\label{thm:existence_and_uniqueness_of_section}
Given a deformation cocycle~$A_\bullet$ of weight~$k$, there exists a unique deformation section~$B$ of weight~$k$ and type~$A_\bullet$.
\end{theorem}

\begin{proof}
As usual, we only need to establish the result in weight~$0$. Uniqueness is given by Theorem~\ref{thm:uniqueness_of_coboundary}, and so it remains to show the existence of a section for every deformation cocycle. This is equivalent to showing that the natural map
\begin{gather*}
  \rmH^1(\Ga,\bbD_k)
\lra
  \rmH^1(\Ga,\bbO_k)
\end{gather*}
is the zero map. Analogously to the proof of Theorem~\ref{thm:extension-of-first-order}, we have an exact sequence of pointed sets
\begin{gather*}
\cdots \lra
  \rmH^1\big(\Ga, \frako_0\big)
\lra
  \rmH^1\big(\Ga, \bbO_0\big\slash (\rho)^{N+1} \big)
\lra
  \rmH^1\big(\Ga, \bbO_0\big\slash (\rho)^{N} \big)
\lra
  0
\tx{,}
\end{gather*}
where we denote by~$\bbO_0\big\slash (\rho)^N$ the image of the natural map
\begin{gather*}
	\bbO_0
\lra
	\cO[\partial_\tau]\llbrkt \rho_1,\ldots,\rho_n\rrbrkt \big\slash (\rho)^N
\tx{.}
\end{gather*}
Via the standard identification of group cohomology with the cohomology of the modular curve with coefficients in a local system~\cite{eilenberg45}, we have that
\begin{gather*}
  \rmH^1\big(\Ga,\frako_0\big)
\cong 
  \rmH^1\big(Y_\Ga, \Omega\big)  \cong 0
\tx{,}
\end{gather*}
as every~$Y_\Ga$ is a Stein manifold~\cite{behnke47} and hence has vanishing higher cohomology groups for sheaf cohomology valued in a line bundle. Hence, we have an isomorphism
\begin{gather*}
  \rmH^1\big(\Ga, \bbO_0\big\slash (\rho)^{N+1} \big)
\cong
  \rmH^1\big(\Ga, \bbO_0\big\slash (\rho)^{N} \big)
\end{gather*}
for every~$N\geq 1$. But
\begin{gather*}
  \rmH^1\big(\Ga, \bbO_0\big\slash (\rho)^{1} \big)
\cong
  \bigoplus_{i=1}^n \rmH^1\big(\Ga, \frako_0\big) = 0
\tx{,}
\end{gather*}
and so~$\rmH^1\big(\Ga, \bbO_0\big\slash (\rho)^{N})\big)=0$ for every~$N\geq 1$. Taking the projective limit, we see that~$\rmH^1\big(\Ga, \bbO_0\big)$ vanishes, and hence
\begin{gather*}
  \rmH^1(\Ga,\bbD_k)
\lra
  \rmH^1(\Ga,\bbO_k)
\end{gather*}
is the zero map.
\end{proof}

As mentioned previously, this approach is entirely non-constructive. We offer the following pair of results demonstrating a method for explicitly constructing a logarithmic section, assuming only that we can find a coboundary realisation for certain linear cocycles.

\begin{proposition}
\label{prop:extending_sections}
Given a pair of a \emph{first order} algebraic deformation and first order deformation section~$(a_\bullet,b)$ of weight~$k$, it is possible to find a pair of a \emph{formal} deformation cycle and deformation section~$(A_\bullet,B)$ such that
\begin{gather*}
  A_\bullet \equiv \exp(a_\bullet) \equiv 1 + a_\bullet
\quad\text{and}\quad
  B \equiv \exp(b) \equiv 1 + b
  \qquad\pmod{(\rho)^2}
\tx{.}
\end{gather*}
\end{proposition}

\begin{proof}
We consider only the weight~$0$ case, as the other cases follow by Corollary~\ref{cor:weight-indep}. 
Working at the level of Lie algebras we use the ansatz
\begin{gather*}
  (A_\bullet,\, B)
=
  \Big(
  \exp\big( a^\Sigma_\bullet \big),\,
  \exp\big( b^\Sigma \big)
  \Big)
\end{gather*}
with
\begin{gather*}
  a^\Sigma_\bullet
=
  \sum a^{(m)}_\bullet\, \rho_1^{m_1} \cdots \rho_n^{m_n}
\in
  \frakd_0^\rho
\tx{,}\quad
  b^\Sigma
=
  \sum b^{(m)}\, \rho_1^{m_1} \cdots \rho_n^{m_n}
\in 
  \frako_0^\rho
\tx{,}
\end{gather*}
where~$m = (m_1,\ldots,m_n)$ runs through~$n$\nbd{}tuples of nonnegative integers, and
\begin{gather*}
  a^\Sigma_\bullet
\equiv
  a_\bullet
\quad\tx{and}\quad
  b^\Sigma
\equiv
  b
  \qquad\pmod{(\rho)^2}
\tx{.}
\end{gather*}
We need to ensure that
\begin{gather*}
  a^\Sigma_{\ga\delta} = \BCH\big( a^\Sigma_\ga\big\|_0\delta,\, a^\Sigma_\delta \big)
\quad\tx{and}\quad
  a^\Sigma_\ga = \BCH\big( b^\Sigma\big\|_0\ga,\, -b^\Sigma \big)
\tx{,}
\end{gather*}
In fact, it suffices to find~$b^\Sigma$ such that
\begin{gather}
\label{eq:thm:extending_sections:prf:cocycle}
  a^\Sigma_\ga
\defeq
  \BCH\big( b^\Sigma\|_0\ga,\, -b^\Sigma \big)
\end{gather}
defines an element of~$\frakd_0^\rho$.

As such, suppose we have constructed the section~$b^\Sigma$ to order~$N$ and consider an index~$m=(m_1,\ldots,m_n)$ with~$|m|=N+1$. Consider the coefficient of~$\rho_1^{m_1}\ldots\rho_n^{m_n}$ of Equation~\eqref{eq:thm:extending_sections:prf:cocycle}. In order for~$a^\Sigma_\ga$ to be an element of~$\frakd_0^\rho$, we require that the left hand side be a quadratic polynomial. Hence, it suffices to find~$b^{(m)}$ such that the coefficient of~$\rho_1^{m_1}\ldots\rho_n^{m_n}$ of the right hand side:
\begin{gather}\label{eq:log-section-construction}
  j_\ga^2(\tau)\,b^{(m)}(\ga\tau)
- b^{(m)}(\tau)
+ \BCH\big( b^\Sigma;\ga \big)^{(m)}(\tau)
\end{gather}
is a quadratic polynomial. Here we use~$\BCH(b^\Sigma;\ga)^{(m)}(\tau)$ to denote the coefficient of the monomial~$\rho_1^{m_1}\cdots\rho_n^{m_n}$ in the non-linear part of the BCH-Formula~\eqref{eq:bch}
\begin{gather*}
  \tfrac{1}{2}
  \big[ b^\Sigma\big\|_0\ga,\, -b^\Sigma \big]
+ \tfrac{1}{12}
  \Big[ b^\Sigma\big\|_0\ga,\, \big[b^\Sigma\big\|_0\ga,\, -b^\Sigma\big] \Big]
+ \cdots
\tx{.}
\end{gather*}
Notably,~$\BCH(b^\Sigma;\ga)^{(m)}$ is defined in terms of lower order terms of~$b^\Sigma$, which we assume to have constructed.

As we want the Expression~\eqref{eq:log-section-construction} to be a quadratic polynomial, it must be annihilated by taking the third derivative. We therefore wish to find~$b^{(m)}$ such that 
\begin{gather}\label{eq:bch-coboundary}
  j_\ga^{-4}(\tau)\frac{\rmd^3b^{(m)}}{\dtau^3}(\ga\tau)
  - \frac{\rmd^3b^{(m)}}{\dtau^3}(\tau)
=
  - \frac{\rmd^3}{\dtau^3}\BCH\big( b^\Sigma;\ga \big)^{(m)}(\tau)
\tx{,}
\end{gather}
where we use Bol's identity to simplify the calculation of the first term.

By Lemma \ref{lem:BCH-cocycle}, the third derivative of~$\BCH(b^\Sigma;\ga)^{\ga}$ satisfies a cocycle condition
\begin{gather*}
  \frac{\rmd^3}{\dtau^3}\BCH\big( b^\Sigma;\ga\delta \big)^{(m)}(\tau)
- j_\delta^{-4}(\tau)\frac{\rmd^3}{\dtau^3}\BCH\big( b^\Sigma;\ga \big)^{(m)}(\delta\tau)
- \frac{\rmd^3}{\dtau^3}\BCH\big( b^\Sigma;\delta \big)^{(m)}(\tau)
=
  0
\end{gather*}
We can therefore view Equation~\eqref{eq:bch-coboundary} as an equation in group cohomology, realising the cocycle~$\frac{\rmd^3}{\dtau^3}\BCH(b^\Sigma;\ga\delta)^{(m)}$ as a coboundary. Hence, the above equation has a solution if and only if~$\frac{\rmd^3}{\dtau^3}\BCH(b^\Sigma;\ga\delta)^{(m)}$ defines the zero element in the (abelian) group cohomology~$\rmH^1(\Ga,\cO)$, where we view~$\cO$ as a~$\Ga$\nbd{}module.

As previously, we may identify group cohomology with sheaf cohomology of the modular curve to conclude that
\begin{gather*}
   \rmH^1(\Ga,\cO)
 \cong 
   \rmH^1(Y_\Ga, \Omega^2) = 0
\tx{.}
\end{gather*}
Thus, we can always realise ~$\frac{\rmd^3}{\dtau^3}\BCH(b^\Sigma;\ga\delta)^{(m)}$ as a coboundary, and hence find~$\frac{\rmd^3}{\dtau^3}b^{(m)}$. Any convergent triple integral of which can then be taken as~$b^{(m)}$, and the existence of such a choice of coboundary is guaranteed by Theorem~\ref{thm:existence_and_uniqueness_of_section}.
\end{proof}

In the next lemma, to simplify notation, we write~$b$ instead of~$b^\Sigma$ from the proof of Proposition~\ref{prop:extending_sections}.

\begin{lemma}\label{lem:BCH-cocycle}
Let~$b\in \frako_0^\rho$ be a finite series containing only terms of total degree at most~$N$, such that~$\BCH(b;\ga)^{(m)}\in\frakd_0^\rho$ for all~$m=(m_1,\ldots,m_n)$ with~$|m|\leq N$ and all~$\ga\in\Ga$. Then for every~$m=(m_1,\ldots,m_n)$ with~$|m|=N+1$,  the third derivative of\/~$\BCH(b;\ga)^{(m)}$ satisfies the cocycle condition
\begin{gather*}
  \frac{\rmd^3}{\dtau^3}\BCH(b;\ga\delta)^{(m)}(\tau)
- j_\delta^{-4}(\tau)\frac{\rmd^3}{\dtau^3}\BCH(b;\ga)^{(m)}(\delta\tau)
- \frac{\rmd^3}{\dtau^3}\BCH(b;\delta)^{(m)}(\tau)
=
  0
\tx{.}
\end{gather*}
\end{lemma}

\begin{proof}
Define~$c_\ga:=\BCH(b\|_0\ga,-b)$. By definition~$c_\bullet$ is a logarithmic coboundary and hence a logarithmic cocycle. In particular
\begin{gather*}c_{\ga\delta} = \BCH\big(c_\ga\big\|_0\delta,c_\delta\big)\end{gather*}
for all~$\ga,\,\delta\in\Ga$. Fix a multi-index~$m=(m_1,\ldots,m_n)$ such that~$|m|=N+1$. 
Expanding out the cocycle equation, we have that
\begin{gather*}c_{\ga\delta}^{(m)}(\tau) = j_\delta^2(\tau) c_\ga^{(m)}(\delta\tau) + c_\delta^{(m)}(\tau) + \BCH\big(c_\ga\big\|_0\delta,c_\delta\big)^{(m)}\tx{.}\end{gather*}
The third term is a Lie polynomial involving only terms of~$c_\bullet$ of order at most~$N$, which by assumption are all elements of~$\frakd_0^\rho$. As~$\frakd_0^\rho$ is closed as a Lie algebra, this implies that 
\begin{gather*}\BCH(c_\ga\big\|_0\delta,c_\delta)^{(m)}\partial_\tau \in\frakd_0^\rho\end{gather*}
and, in particular~$\BCH(c_\ga\big\|_0\delta,c_\delta)^{(m)}$ is a quadratic polynomial in~$\tau$. Hence, the non-linear part with respect to~$\partial_\tau$ of the above equation is annihilated by the third derivative. Employing Bol's identity to simplify calculations, we have that
\begin{gather*}
  \frac{\rmd^3 c_{\ga\delta}^{(m)}}{\dtau^3}(\tau)
=
  j_\delta^{-4}(\tau)\frac{\rmd^3 c_{\ga}^{(m)}}{\dtau^3}(\delta\tau) + \frac{\rmd^3 c_{\delta}^{(m)}}{\dtau^3}(\tau)
\end{gather*}
from which the claim follows.
\end{proof}

Extending Proposition~\ref{prop:extending_sections}, we can then, up to determining linear coboundaries, explicitly construct the unique deformation section associated to a deformation cocycle.

\begin{corollary}%
\label{cor:constructing-sections}
Given a deformation cocycle~$A_\bullet$ of weight~$k$,we can logarithmically construct the unique deformation section~$B$ of weight~$k$ and type~$A_\bullet$.
\end{corollary}
\begin{proof}
As usual, we only need to establish the result in weight~$0$. Let~$a_\bullet\in\frakd_0^\rho$ be such that~$A_\bullet=\exp(a_\bullet)$ and suppose we have constructed~$b\in\frako_0^\rho$ up to order~$N$ such that
\begin{gather*}
  a_\ga \equiv \BCH\big(b\big\|_0\ga,-b\big) \;\pmod{(\rho)^{N+1}}
\tx{.}
\end{gather*}
By the methods of Proposition~\ref{prop:extending_sections}, it is possible to extend~$b$ to a~$\tilde{b}\in\frako_0^\rho$ of order~$N+1$ such that
\begin{gather*}
  a_\ga \equiv \BCH\big(\tilde{b}\big\|_0\ga,-\tilde{b}\big) \;\pmod{(\rho)^{N+1}}
\end{gather*}
and
\begin{gather*}
  c_\ga \defequiv \BCH\big( \tilde{b}\big\|_0\ga,-\tilde{b}\big) \in\frakd_0^\rho \;\pmod{(\rho)^{N+2}}
\tx{.}
\end{gather*}
In particular, for every index~$m=(m_1,\ldots,m_n)$ of order~$N+1$, we have that
\begin{gather*}c_\ga^{(m)} = \tilde{b}^{(m)}\big\|_0\ga - \tilde{b}^{(m)} + \BCH(b;\ga)^{(m)}\end{gather*}
is a quadratic polynomial.

By construction, we must have that ~$a_\bullet^{(m)}-c_\bullet^{(m)}$ is a linear 1\nbd{}cocycle for the standard action of~$\Ga$ on quadratic polynomials. In particular, we know how to realise such cocycles as coboundaries, i.e. there exists~$d^{(m)}=f(\tau)\partial_\tau$ such that
\begin{gather*}a_\bullet^{(m)}-c_\bullet^{(m)} = d^{(m)}\big\|_0\ga - d^{(m)}\tx{.}\end{gather*}
So by defining~$b^{(m)}:=\tilde{b}^{(m)} + d^{(m)}$, we obtain an extension of~$b$ to order~$N+1$ such that
\begin{gather*}
  a_\ga \equiv \BCH\big( b\big\|_0\ga,-b \big) \;\pmod{(\rho)^{N+2}}
\tx{.}
\end{gather*}
As we can always construct~$b$ to order 1, as in Section \ref{sec:first-order-deformations}, this allows us to recursively construct the desired section. 
\end{proof}

\subsection{A canonical deformation of motivic origin}%
\label{ssec:canonical_deformation}

In order to give a canonical example of a deformation cocycle and section, we draw on the work of Brown on the motivic fundamental group of the moduli space~$\cM_{1,1}$~\cite{brown17}. We recall here (a modified version of) his notation. Fix a basis~$\cB$ of the weight graded space of all modular forms~$\rmM(\Ga)$ and denote by~$\cB_k=\{h_1,\ldots,h_{\ell_k}\}$ the weight~$k$ component of~$\cB$. We assume that~$\cB_k$ contains a basis for the space of weight~$k$ cusp forms. Define~$M^\vee_k$ to be the dual of the weight~$k$ component of~$\rmM(\Ga)$, with basis~$\{A_h\}_{h\in\cB_k}$.

We denote by~$V_{k-2}$ the right~$\Ga$\nbd{}module spanned by polynomials of degree at most~$k-2$ with~$\Ga$\nbd{}action
\begin{gather*}
  p \big\|\ga (\tau) \defeq j_\ga(\tau)^{k-2}p(\tau)
\tx{,}
\end{gather*}
and define the graded right~$\Ga$\nbd{}module
\begin{gather*}
  M^\vee \defeq \bigoplus_{k\geq 0 }M^\vee_k \otimes V_{k-2}
\tx{.}
\end{gather*}
Finally denote by~$\cO\llangle M^{\vee}\rrangle$ the ring of formal power series in~$M^\vee$ with coefficients in~$\cO$. Note that we can think of an element S of~$\cO\llangle M^{\vee}\rrangle$ as linear functions
\begin{gather*}
  \rmM_{k_1}\otimes \cdots \otimes \rmM_{k_r}
\lra
  \cO\otimes V_{k_1-2}\otimes\cdots \otimes V_{k_r-2}
\end{gather*}
by saying that~$S(h_1,\ldots,h_r)$ is equal to the coefficient of~$A_{h_1}\otimes\cdots \otimes A_{h_r}$ in~$S$.

A variation of Brown's approach~\cite{brown17} lets us define an element of~$\cO\llangle M^{\vee}\rrangle$ via iterated integrals as follows. Let 
\begin{gather*}
  \Omega(x,z) 
= 
  \sum_{k\geq 0}\sum_{h\in\cB_k} A_h \otimes \left( (2\pi i)^{k-1}h(z) (x - z)^{k-2} \,\dz\right)
\tx{,}
\end{gather*}
and define~$I(\tau)\in\cO\llangle M^{\vee}\rrangle$ as the iterated integral
\begin{gather*}
  I(\tau;x_1,x_2,\ldots) 
= 
  1 + \int_\tau^\infty \Omega(x_1,z) + \int_{\tau}^\infty \Omega(x_2,z)\Omega(x_1,z) + \cdots
\tx{.}
\end{gather*}
This is a regularised iterated integral, but we suppress reference to tangential basepoints for ease of notation. By Lemma 5.1 of~\cite{brown17}
\begin{gather}
\label{eq:def:canonical_cocycle}
  \cC_\ga(\tau) 
\defeq 
  \big(I(\ga\tau) \big\|\ga \big)^{-1} I(\tau)
\end{gather}
is a~$1$-cocycle such that
\begin{gather*}
 \cC_\ga(\tau)(h_1,\ldots,h_r) \in \bbC\otimes V_{k_1-2}\otimes\cdots \otimes V_{k_r-2}
 \tx{.}
 \end{gather*}
The cocycle~$\cC_\bullet$ is referred to as the totally holomorphic canonical cocycle. Both~$\cC_\ga(\tau)$ and~$I(\tau)$ are examples of series which are \textit{grouplike}. As the definition of this grouplike property requires some technical machinery that is not needed elsewhere, we refer the reader to Brown's discussion for further details. 

We then define a collection of maps
\begin{gather*}d_{k_1,\ldots,k_r}:\cO\otimes V_{k_1-2}\otimes\cdots \otimes V_{k_r-2}\lra \bbC[\partial_\tau]\end{gather*}
 as follows:
\begin{gather*}
d_{k_1,\ldots,k_r}(f\otimes p_1\otimes\cdots\otimes p_r) = \begin{cases} 0 \text{ if } k_i\neq 4\text{ for some }i,\\
f(\tau)\left( p_1(\tau)\partial_\tau \right)\cdots \left(p_r(\tau)\partial_\tau\right)\text{ otherwise.}\end{cases}\end{gather*}

The composition~$dS$ of an element~$S\in \cO\llangle M^{\vee}\rrangle$, viewed as a collection of linear maps, with the~$d_{k_1,\ldots,k_r}$ defines a collection of linear maps, which can then be viewed as a power series
\begin{gather*}
  dS \in \cO[\partial_\tau] \bigllangle M_4^\vee \bigrrangle
\tx{.}
\end{gather*}
Denote by 
\begin{gather*}
  \pi :
  \cO[\partial_\tau] \bigllangle M_4^\vee \bigrrangle
\lra
  \cO[\partial_\tau]\llbrkt\rho_1,\ldots,\rho_n\rrbrkt
\end{gather*}
the map induced by sending~$A_{h_i}$ to~$\rho_i$ for every weight~$4$ cusp form~$h_i\in \cB_4$, and~$A_h$ to~$0$ for all other~$h$. By composition, we obtain a map 
\begin{gather}
\label{defeq:motivic-projection}
\begin{aligned}
   D: \cO\bigllangle M^\vee \bigrrangle
 &\lra
   \cO[\partial_\tau] \llbrkt \rho_1,\ldots,\rho_n\rrbrkt
\tx{,}\\
   S
 &\lmto
   \pi(dS)
 \tx{.}
\end{aligned}
 \end{gather}
 
The proof of the following lemma is somewhat technical, but it is essential for our discussions.
 
 \begin{lemma}\label{lem:motivic-series-in-o}
 The map~$D$  takes grouplike elements of\/~$\bbC\llangle M^\vee \rrangle$ to elements in the pro-uni\-potent group~$\bbD_0$ defined by Equation \eqref{defeq:groups_od}.
 \end{lemma}

 \begin{proof}
 For a pro-nilpotent Lie algebra~$\frakg$, the Lie group~$\exp(\frakg)$ is isomorphic to the space of grouplike elements in the completed universal enveloping algebra~$\widehat{\rmU}(\frakg)$. In particular, as~$\bbD_0$ is defined as the exponential of the pro-nilpotent Lie algebra~$\frakd_0^\rho$, it may be identified with the space of grouplike elements of~$\widehat{\rmU}(\frakd_0^\rho)$.

Next note that the restriction of~$D$ to~$\bbC\llangle M^\vee \rrangle$ is induced by a map of Lie algebras
\begin{gather}
\begin{aligned}
D:\Lie M^\vee &\lra \frakd_0^\rho\tx{,}\\
A_h\otimes p &\lmto \pi\big( A_h\otimes d_k(p) \big)
\tx{,}
\end{aligned}
\end{gather}
where~$\Lie M^\vee$ is the free Lie algebra on~$M^\vee$, and~$h$ is a weight~$k$ modular form. Hence,~$D$ factors via the (completed) universal enveloping algebra~$\widehat{\rmU}(\frakd_0^\rho)$. Furthermore, as~$D$ is induced by a map of Lie algebras, the restriction of~$D$ to grouplike elements factors through the space of grouplike elements of~$\widehat{\rmU}(\frakd_0^\rho)$. As this is isomorphic to~$\bbD_0$, we must have that~$D$ takes grouplike elements to elements of~$\bbD_0$.
 \end{proof}

\begin{theorem}
\label{thm:canonical-cocycle}
The image of Brown's totally holomorphic canonical cocycle~$\cC_\bullet$  under the map~$D$ defines a deformation cocycle of weight~$0$. In particular, we obtain a ``totally holomorphic'' canonical deformation cocycle, of motivic origin, defined in terms of classical multiple modular values.
\end{theorem}
\begin{proof}
By Lemma \ref{lem:motivic-series-in-o} ~$D(\cC_\bullet)\in\bbD_0$. We will first show that~$D$ is~$\Ga$\nbd{}equivariant, or equivalently that~$d_{k_1,\ldots,k_r}$ is~$\Ga$\nbd{}equivariant. When the map~$d_{k_1,\ldots,k_r}$ is non-zero, 
\begin{align*}
& d_{k_1,\ldots,k_r} \big(\big(f(\tau)\otimes p_1(\tau)\otimes\cdots \otimes p_r(\tau) \big)\|\ga\big)\\[.2\baselineskip]
={}& d_{k_1,\ldots,k_r}\big( f(\ga\tau) \otimes j^2_\ga(\tau)p_1(\ga\tau)\otimes\cdots\otimes j^2_\ga(\tau) p_r(\ga\tau)\big)\\
={}& f(\ga\tau) \prod_{i=1}^r \big( j^2_\ga(\tau)p_i(\ga\tau) \partial_\tau\big)
=    f(\ga\tau)\prod_{i=1}^r \big( p_i(\ga\tau)\partial_{\ga\tau}\big)
=    \Big(f(\tau)\prod_{i=1}^r\big(p_i(\tau)\partial_\tau\big)\Big)\big\|_0 \ga\\ 
={}& d_{k_1,\ldots,k_r}\big(f(\tau)\otimes p_1(\tau)\otimes\cdots\otimes p_r(\tau)\big)\big\|_0 \ga
\tx{.}
\end{align*}
Thus~$D$ is~$\Ga$\nbd{}equivariant. Next note that, when restricted to~$\bbC\llangle M^\vee\rrangle$,~$D$ is a homomorphism as
\begin{align*}
&
  d_{k_1,\ldots,k_r,\ell_1,\ldots,\ell_s}
  \big( xy\otimes p_1\otimes\cdots\otimes p_r\otimes q_1\otimes \cdots\otimes q_s \big)
\\
={}&
  d_{k_1,\ldots,k_r}(x\otimes p_1\otimes\cdots\otimes p_r)\,
  d_{\ell_1,\ldots,\ell_s}(y\otimes q_1\otimes\cdots\otimes q_s)
\tx{.}
\end{align*}
Hence
\begin{gather*}
  D(\cC_{\ga\delta})
=
  D\big( \cC_\ga\big\|\delta\cdot\cC_\delta \big)
=
  D(\cC_\ga)\big\|_0\delta\cdot D(\cC_\delta)
\tx{,}
\end{gather*}
and so~$D(\cC_\bullet)$ is a cocycle.
\end{proof}

\begin{remark}%
\label{rem:canonical-section}
Note that 
\begin{gather*}
  D: \cO[\partial_\tau] \bigllangle M^\vee \bigrrangle
\lra
  \cO[\partial_\tau] \llbrkt \rho_1,\ldots,\rho_n\rrbrkt
\end{gather*}
does not respect products, and as such, we do not expect 
\begin{gather*}D(\cC_\ga) = (DI)^{-1}\big\|\ga\cdot DI.\end{gather*}
Indeed, we have seen this in Lemma~\ref{lem:second-order-solutions}, in the appearance of~$\Lambda_\tau(h,h;3,0)$ in our solutions --- Brown's formalism never produces multiple modular values~$\Lambda_\tau(h_1,h_2;n_1,n_2)$ associated to cusp forms of weight~$4$ with~$n_i>2$. However, as noted in Theorem~\ref{thm:existence_and_uniqueness_of_section}, there exists a unique deformation section associated to the totally holomorphic canonical cocycle~$D(\cC_\bullet)$, which we can compute via Corollary~\ref{cor:constructing-sections}. As such, we can consider this a canonical deformation section of motivic origin.
\end{remark}

While~$D\cC_\bullet = D(\cC_\bullet)$ has the advantage of being explicit, it is not truly canonical, in the sense that we cannot relate it directly with arbitrary formal deformation cocycles. Via the Tannakian theory of relative completion, Brown gives an abstract definition of a motivic canonical cocycle~$\cA_\bullet$ in~\cite{brown17}[Definition 15.4].This produces a grouplike element of~$\mathbb{C}\llangle P^\vee\rrangle$ where~$P^\vee$ is the graded right~$\Ga$\nbd{}module
\begin{gather*}
  P^\vee
\defeq
 \bigoplus_{k\geq 0} P_k^\vee\otimes V_{k-2}
\tx{,}
\end{gather*}
and~$P_k^\vee$ is the dual of~$\rmH^1(\Ga,V_{k-2})$, spanned by symbols~$A_h$ for every~$h\in\cB_k$ and~$A_{\overline{h}}$ for every cusp form~$h\in\cB_k$. We can extend the map ~$D$ defined in Equation~\eqref{defeq:motivic-projection} to to a map
\begin{gather*}
\begin{aligned}
   D: \cO\bigllangle P^\vee \bigrrangle
 &\lra
   \cO[\partial_\tau] \llbrkt \rho_1,\ldots,\rho_n,\rho_{n+1},\ldots\rrbrkt
\tx{,}\\
   S
 &\lmto
   \pi(dS)
 \tx{.}
\end{aligned}
\end{gather*}
by sending~$A_{\overline{h}_i}$ to~$\rho_{n+i}$. The arguments of Lemma \ref{lem:motivic-series-in-o} and Theorem \ref{thm:canonical-cocycle} apply verbatim, from which we deduce the existence of a canonical deformation cocycle. Alongside Lemma 15.5 of~\cite{brown17}, we conclude the following.

\begin{theorem}
\label{thm:true-canonical-cocycle}
The image of Brown's canonical cocycle~$\cA_\bullet$  under the map~$D$ defines a deformation cocycle of weight~$0$. In particular, we obtain a canonical deformation cocycle, of motivic origin, defined in terms of periods of the motivic fundamental group of the modular curve. Furthermore
\begin{gather*}
 D\cC_\bullet(\rho_1,\ldots,\rho_n)
=
 D\cA_\bullet(\rho_1,\ldots,\rho_n,0,\ldots,0)
\tx{.}
\end{gather*}
\end{theorem}

The cocycle~$\cA_\bullet$ satisfies a tangency condition, which for our purposes means that the deformation cocycle~$D\cA_\bullet$ has first order coefficients that define a basis of~$\rmH^1(\Ga,\frakd_0)$. As such, we can prove an analogue of Corollary 11.7 of~\cite{brown17}, following the more explicit arguments of~\cite{brown14}. We first define a monoid acting on the space of cocycles, which can easily be checked to be well defined.

\begin{definition}
\label{def:endomorphism-group}
Denote by~$\bbD_0\rtimes \End(\bbD_0)^\Ga$ the monoid given by the Cartesian product of~$\bbD_0$ with its~$\Ga$\nbd{}equivariant endomorphisms, equipped with product
\begin{gather*} 
  (C_1,\Phi_1)\cdot(C_2,\Phi_2)
\defeq
  \big( C_1\Phi_1(C_2),\, \Phi_1\Phi_2 \big)
\tx{.}
\end{gather*}
Define an action of~$\bbD_0\rtimes \End(\bbD_0)^\Ga$ on~$\rmZ^1(\Ga,\bbD_0)$ by
\begin{gather}%
\label{eq:monoid-action}
  \big((C,\Phi) A \big)_\ga 
\defeq
  C\big\|_0\ga\cdot \Phi(A_\ga)\cdot C^{-1}
\tx{.}
\end{gather}
\end{definition}

\begin{theorem}
\label{thm:transitive-action}
The orbit of\/~$D\cA_\bullet$ under the action~\eqref{eq:monoid-action} is the entirety of\/~$\rmZ^1(\Ga,\bbD_0)$.
\end{theorem}

\begin{proof}
We first note that describing an endomorphism~$\Phi$ of~$\bbD_0$ is equivalent to describing an endomorphism~$\phi$ of~$\frakd_0^\rho$. Furthermore, as~$\frakd_0^\rho$ is generated as a pro-nilpotent Lie algebra by~$\frakd_0$, such an endomorphism is uniquely determined by a map
\begin{gather*}
  \phi:\frakd_0
\lra
  \frakd_0^\rho
\tx{.}
\end{gather*}
We denote by 
\begin{gather*}
  \phi_{m}: \frakd_0
\lra
  \frakd_0
\end{gather*}
the linear map given by taking the coefficient of~$\rho^m$ in the image of~$\phi$.

Denote by~$\underline{a}_{\bullet,i}$ the coefficient of~$\rho_i$ in~$D\cA_\bullet$. Suppose we are given~$A_\bullet\in\rmZ^1(\Ga,\bbD_0)$, and suppose further that we have constructed~$C\in\bbD_0$ and~$\Ga$\nbd{}equivariant~$\phi$ to order~$N$ such that
\begin{gather*}
 A_\bullet
\equiv
 \Big( \big(C,\exp(\phi)\big)D\cA \Big)_\bullet \;\pmod{(\rho)^{N}}
\tx{,}
\end{gather*}
for some~$N\geq 1$. Let~$m=(m_1,\ldots,m_n)$ be a multi-index such that~$|m|=N$. Considering the coefficient of~$\rho^m$, we see that in order to extend the above congruence to order~$N$, we need to find~$c_m$ and~$\phi_m$ such that
\begin{gather*}
  a_{\ga,m} 
= 
  c_m\big\|_0\ga + \sum \phi_m(\underline{a}_{\ga,i}) - c_m + r_{\ga,m}
\tx{,}
\end{gather*}
where we denote by~$r_{\ga,m}$ all the terms involving products of lower order terms. Considering these as elements of (abelian) group cohomology, we have that
\begin{gather*}
\begin{aligned}
  \delta^1\big( r_{\bullet,m}\big)(\ga,\delta)
&=
  - \sum_{\substack{p+q=m\\ p,q\neq 0}}
  \Big( \big(C,\exp(\phi)\big) D\cA \Big)_{\ga,p} \Big\|_k\delta
  \mathbin{\,\cdot\,}
  \Big( \big(C,\exp(\phi)\big) D\cA \Big)_{\delta,q}
\\
  &=- \sum_{\substack{p+q=m\\ p,q\neq 0}} a_{\ga,p}\big\|_k\delta\cdot a_{\delta,q}
=
  \delta^1\big(a_{\bullet,m}\big)(\ga,\delta)
\tx{,}
\end{aligned}
\end{gather*}
and hence~$a_{\bullet,m} -r_{\bullet,m}$ is an abelian ~$\Ga$\nbd{}cocycle. Then, as~$\underline{a}_{\bullet,i}$ form a basis for cohomology, it is possible to choose~$\phi_m(\underline{a}_{\bullet,i})$ such that
\begin{gather*}
  a_{\bullet,m} - r_{\bullet,m} - \sum\phi_m(\underline{a}_{\bullet,i})
\end{gather*}
vanishes in cohomology, and hence there exists~$c_m\in\frakd_0$ such that the above equality holds. In fact, we can assume that~$\phi_m(\underline{a}_{\bullet,i})$ is a scalar multiple of~$\underline{a}_{\bullet,i}$, ensuring that~$\phi_m$ is~$\Ga$\nbd{}equivariant. 

Then, as all elements of~$\bbD_0$ have constant term~$1$, we can construct the desired~$(C,\Phi)$ inductively.
\end{proof}

\ifbool{nobiblatex}{%
  \bibliographystyle{alpha}%
  \bibliography{bibliography.bib}%
  \addcontentsline{toc}{section}{References}
  \markright{References}
}{%
  \vspace{1.5\baselineskip}
  \phantomsection
  \addcontentsline{toc}{section}{References}
  \markright{References}
  \label{sec:references}
  \sloppy
  \printbibliography[heading=none]%
}

\Needspace*{3\baselineskip}
\noindent%
\rule{\textwidth}{0.15em}
\\

{\noindent\small
Adam Keilthy\\
Chalmers tekniska högskola och G\"oteborgs Universitet\\
Institutionen f\"or Matematiska vetenskaper\\
SE-412 96 G\"oteborg, Sweden\\
E-mail: \url{keilthy@chalmers.se}\\
Homepage: \url{https://www.maths.tcd.ie/~keilthya}
}\vspace{.5\baselineskip}

{\noindent\small
Martin Raum\\
Chalmers tekniska högskola och G\"oteborgs Universitet\\
Institutionen f\"or Matematiska vetenskaper\\
SE-412 96 G\"oteborg, Sweden\\
E-mail: \url{martin@raum-brothers.eu}\\
Homepage: \url{https://martin.raum-brothers.eu}
}%

\ifdraft{%
\listoftodos%
}

\end{document}

%% file: packages.tex
\usepackage{etoolbox}
\usepackage{ifdraft}
\usepackage{iftex}

\ifluatex
\usepackage[utf8]{luainputenc}
\else
\usepackage[utf8]{inputenc}
\fi
\usepackage[T1]{fontenc}

\usepackage[english]{babel}

\usepackage{lmodern}
\usepackage{fourier}

\usepackage{amsfonts}
\usepackage{mathrsfs} %
\usepackage{dsfont}

\usepackage{amssymb}

\usepackage{amsmath}
\usepackage{mathtools}
\usepackage[thmmarks, amsmath, amsthm]{ntheorem}
\usepackage{nccmath}    %

\usepackage[draft=false]{scrlayer-scrpage}
\usepackage{needspace}

\providebool{nobiblatex}
\boolfalse{nobiblatex}
\ifbool{nobiblatex}{%
}{%
  \usepackage[backend=biber,style=numeric-comp,giveninits,url=false,sortcites=true,maxbibnames=99]{biblatex}
  \addbibresource{bibliography.bib}
}

\usepackage[final,
            pdfborder={0 0 0}, colorlinks=true,
            linkcolor=BrickRed, citecolor=ForestGreen, urlcolor=RoyalBlue]{hyperref}

\usepackage[cmyk,dvipsnames]{xcolor}

\ifdraft{%
\usepackage[textsize=scriptsize,colorinlistoftodos]{todonotes}
\usepackage{lineno}
\usepackage{scrtime}
}{}

\usepackage{booktabs}
\usepackage[inline]{enumitem}
\usepackage{lettrine}
\usepackage{tikz-cd}
\usepackage{url}

%% file: layout.tex
\KOMAoptions{DIV=12}

\ifdraft{%
\linenumbers
\newcommand*\patchAmsMathEnvironmentForLineno[1]{%
  \expandafter\let\csname old#1\expandafter\endcsname\csname #1\endcsname
  \expandafter\let\csname oldend#1\expandafter\endcsname\csname end#1\endcsname
  \renewenvironment{#1}%
     {\linenomath\csname old#1\endcsname}%
     {\csname oldend#1\endcsname\endlinenomath}}%
\newcommand*\patchBothAmsMathEnvironmentsForLineno[1]{%
  \patchAmsMathEnvironmentForLineno{#1}%
  \patchAmsMathEnvironmentForLineno{#1*}}%
\AtBeginDocument{%
\patchBothAmsMathEnvironmentsForLineno{equation}%
\patchBothAmsMathEnvironmentsForLineno{align}%
\patchBothAmsMathEnvironmentsForLineno{flalign}%
\patchBothAmsMathEnvironmentsForLineno{alignat}%
\patchBothAmsMathEnvironmentsForLineno{gather}%
\patchBothAmsMathEnvironmentsForLineno{multline}%
}}

\clearmainofpairofpagestyles
\automark[section]{subsection}

\renewcommand{\subsectionmark}[1]{}

\lohead{{\small\normalfont \headertitle}}
\rohead{{\small \headerauthors}}

\RedeclareSectionCommand[%
  font=\Large\sffamily\bfseries,%
  beforeskip=1\baselineskip,%
  afterskip=0.5\baselineskip,%
  indent=0em,%
  afterindent=false,%
  tocbeforeskip=0.3\baselineskip plus 1pt minus 1pt%
  ]{section}

\RedeclareSectionCommands[%
  font=\normalfont\bfseries,%
  beforeskip=3pt,%
  afterskip=-1em%
  ]{subsection,subsubsection}

\RedeclareSectionCommands[%
  font=\normalfont\itshape,%
  beforeskip=1pt,%
  afterskip=-1em,%
  indent=0pt,%
  ]{paragraph}

\ifbool{nobiblatex}{}{%
  \AtEveryBibitem{\clearfield{doi}}
  \AtEveryBibitem{\clearfield{isbn}}
  \AtEveryBibitem{\clearfield{issn}}
  \AtEveryBibitem{\clearfield{pages}}
  \AtEveryBibitem{\clearlist{language}}

  \setlength\bibitemsep{3pt}
  \renewcommand*{\bibfont}{\footnotesize}
  \renewbibmacro*{in:}{}
  \DeclareFieldFormat
    [article,inbook,incollection,inproceedings,patent,thesis,unpublished,misc]
    {title}{#1}
}

\newenvironment{enumerateroman}{
\begin{enumerate}[label=(\roman*),%
  leftmargin=2.5em,itemindent=0pt,%
  labelindent=.5em,labelwidth=1.5em,labelsep=!,%
  nosep]
}{
\end{enumerate}
}

\newenvironment{enumeratearabic*}{
\begin{enumerate*}[label=(\arabic*)] %
}{
\end{enumerate*}
}

\newenvironment{enumerateroman*}{
\begin{enumerate*}[label=(\roman*)] %
}{
\end{enumerate*}
}

\numberwithin{equation}{section}

\theoremnumbering{arabic}
\newtheorem{theoremcounter}{theoremcounter}[section]
\theoremnumbering{Alph}
\newtheorem{maintheoremcounter}{maintheoremcounter}

\theoremstyle{plain}

\newtheorem{corollary}[theoremcounter]{Corollary}
\newtheorem{lemma}[theoremcounter]{Lemma}

\newtheorem{proposition}[theoremcounter]{Proposition}
\newtheorem{theorem}[theoremcounter]{Theorem}

\theoremstyle{plain}

\newtheorem{maintheorem}[maintheoremcounter]{Theorem}

\theoremstyle{definition}

\newtheorem{definition}[theoremcounter]{Definition}
\newtheorem{example}[theoremcounter]{Example}

\theoremstyle{remark}

\newtheorem{remark}[theoremcounter]{Remark}

\theoremstyle{nonumberremark}

\newtheorem{mainexample}{Example}

%% file: shortcuts.tex
\newcommand{\tx}{\text}
\newcommand{\nbd}{\nobreakdash-\hspace{0pt}}

\newcommand{\texpdf}[2]{\texorpdfstring{#1}{#2}}

\makeatletter
\newcommand{\writelabel}[1]{#1\def\@currentlabel{#1}}
\makeatother

\newcommand{\minwidthmathbox}[2]{%
  \mathmakebox[{\ifdim#1<\width\width\else#1\fi}]{#2}%
}

\newcommand{\tbf}{\bfseries}

\newcommand{\bbC}{\ensuremath{\mathbb{C}}}
\newcommand{\bbD}{\ensuremath{\mathbb{D}}}

\newcommand{\bbH}{\ensuremath{\mathbb{H}}}

\newcommand{\bbO}{\ensuremath{\mathbb{O}}}
\newcommand{\bbP}{\ensuremath{\mathbb{P}}}

\newcommand{\bbR}{\ensuremath{\mathbb{R}}}

\newcommand{\bbZ}{\ensuremath{\mathbb{Z}}}

\newcommand{\cA}{\ensuremath{\mathcal{A}}}
\newcommand{\cB}{\ensuremath{\mathcal{B}}}
\newcommand{\cC}{\ensuremath{\mathcal{C}}}

\newcommand{\cE}{\ensuremath{\mathcal{E}}}

\newcommand{\cJ}{\ensuremath{\mathcal{J}}}

\newcommand{\cL}{\ensuremath{\mathcal{L}}}
\newcommand{\cM}{\ensuremath{\mathcal{M}}}

\newcommand{\cO}{\ensuremath{\mathcal{O}}}
\newcommand{\cP}{\ensuremath{\mathcal{P}}}

\newcommand{\cT}{\ensuremath{\mathcal{T}}}

\newcommand{\cX}{\ensuremath{\mathcal{X}}}
\newcommand{\cY}{\ensuremath{\mathcal{Y}}}

\newcommand{\frakb}{\ensuremath{\mathfrak{b}}}

\newcommand{\frakd}{\ensuremath{\mathfrak{d}}}

\newcommand{\frakg}{\ensuremath{\mathfrak{g}}}

\newcommand{\frakj}{\ensuremath{\mathfrak{j}}}

\newcommand{\frako}{\ensuremath{\mathfrak{o}}}
\newcommand{\frakp}{\ensuremath{\mathfrak{p}}}

\newcommand{\rmd}{\ensuremath{\mathrm{d}}}

\newcommand{\rmH}{\ensuremath{\mathrm{H}}}

\newcommand{\rmL}{\ensuremath{\mathrm{L}}}
\newcommand{\rmM}{\ensuremath{\mathrm{M}}}

\newcommand{\rmS}{\ensuremath{\mathrm{S}}}
\newcommand{\rmT}{\ensuremath{\mathrm{T}}}
\newcommand{\rmU}{\ensuremath{\mathrm{U}}}

\newcommand{\rmZ}{\ensuremath{\mathrm{Z}}}

\newcommand{\wtd}{\widetilde}
\newcommand{\ov}{\overline}

\newcommand{\defcol}{\mathrel{:}}
\newcommand{\defeq}{\mathrel{:=}}

\newcommand{\condcol}{\mathrel{:}}

\newcommand{\llbrkt}{\llbracket}
\newcommand{\rrbrkt}{\rrbracket}

\newcommand{\mrelspace}[1]{\mathrel{\mspace{#1}}}

\let\rightarroworig\rightarrow
\renewcommand{\rightarrow}
  {\protect\relbar\mrelspace{-9.7mu}\rightarroworig}

\let\leftarroworig\leftarrow
\renewcommand{\leftarrow}
  {\protect\leftarroworig\mrelspace{-9.7mu}\relbar}

\renewcommand{\longrightarrow}
  {\protect\relbar\mrelspace{-3.2mu}\relbar\mrelspace{-9.5mu}\rightarroworig}

\newcommand{\longtwoheadrightarrow}
  {\protect\relbar\mrelspace{-3mu}\rightarroworig\mrelspace{-15mu}\rightarroworig}

\makeatletter
\@ifundefined{xlongrightarrow}{%

}{%

}
\makeatother

\makeatletter
\@ifundefined{xlongleftarrow}{%

}{%

}
\makeatother

\newcommand{\ra}{\rightarrow}

\newcommand{\lra}{\longrightarrow}

\newcommand{\lthra}{\longtwoheadrightarrow}

\newcommand{\lmto}{\longmapsto}

\renewcommand{\Im}{\mathrm{Im}}

\renewcommand{\pmod}[1]{\;(\mathrm{mod}\, #1)}

\renewcommand{\ker}{\operatorname{ker}}

\newenvironment{psmatrix}{\left(\begin{smallmatrix}}{\end{smallmatrix}\right)}

\newcommand{\T}{\rmT}

\newcommand{\ZZ}{\ensuremath{\mathbb{Z}}}

\newcommand{\RR}{\ensuremath{\mathbb{R}}}
\newcommand{\CC}{\ensuremath{\mathbb{C}}}

\newcommand{\SL}[1]{\ensuremath{\mathrm{SL}_{#1}}}

\newcommand{\U}[1]{\ensuremath{\mathrm{U}_{#1}}}

\newcommand{\HS}{\ensuremath{\mathbb{H}}}

%% file: shortcuts_this.tex
\newcommand{\defequiv}{\mathrel{:\equiv}}

\newcommand{\Ga}{\Gamma}
\newcommand{\ga}{\gamma}
\newcommand{\ord}{\operatorname{ord}}

\DeclareMathOperator{\Spec}{Spec}

\DeclareMathOperator{\BCH}{BCH}

\DeclareMathOperator{\Der}{Der}
\DeclareMathOperator{\Lie}{Lie}
\DeclareMathOperator{\End}{End}

\newcommand{\llangle}{\mathopen{\langle\mspace{-4mu}\langle}}
\newcommand{\rrangle}{\mathclose{\rangle\mspace{-4mu}\rangle}}
\newcommand{\bigllangle}{\mathopen{\big\langle\mspace{-5mu}\big\langle}}
\newcommand{\bigrrangle}{\mathclose{\big\rangle\mspace{-5mu}\big\rangle}}

\newcommand{\dz}{\rmd\!z}
\newcommand{\dt}{\rmd\!t}
\newcommand{\dx}{\rmd\!x}
\newcommand{\dtau}{\rmd\!\tau}